\documentclass[12pt,reqno,english,empty]{amsart}
\usepackage{array}
\usepackage{frcursive} 
\usepackage[colorlinks=true,
linkcolor=blue,
citecolor=blue,
]{hyperref}
\usepackage[T1]{fontenc}
\usepackage{mathrsfs}
\usepackage{amsmath,amsthm,amssymb}
\usepackage{latexsym}
\usepackage{enumerate}
\usepackage{mathrsfs}
\usepackage{stmaryrd}
\usepackage{amsopn}
\usepackage{amsmath}
\usepackage{amssymb}
\usepackage{amsfonts}
\usepackage{soul}
\usepackage{amsbsy}
\usepackage{amscd,indentfirst,epsfig}
 \usepackage{amsfonts,amsmath,latexsym,amssymb,verbatim,amsbsy}
\usepackage{amsthm}
\usepackage{colordvi}
\usepackage{pstricks,bm}
\usepackage{subfigure}

\usepackage{url}

\newcommand{\abs}[1]{\left\vert#1\right\vert}
\newcommand{\ap}[1]{\left\langle#1\right\rangle}
\newcommand{\norm}[1]{\left\Vert#1\right\Vert}

\def\R{\mathbb{R}}

\def\Q{\mathcal{Q}}

\def \D {\mathscr{D}}
\def \leq {\leqslant}
\def \geq {\geqslant}

\def \d {\,\mathrm{d} }

\def \ds {\,\mathrm{d}s }

\def \v {{v}}

\def \vb {\v_{\ast}}

\def \dD {\mathbb{D}}

\def \wp {\bm{\ell}}

\def\ds {\displaystyle}

\numberwithin{equation}{section}

\newcommand{\pa}[1]{\left(#1 \right)}
\newcommand{\avg}[1]{\langle #1 \rangle}

\newcommand{\Lin}{\mathcal{L}}
\newcommand{\sph}{\mathbb{S}^{d-1}}
\newcommand{\M}{\mathcal{M}}
\newtheorem{thm}{Theorem}[section]
\newtheorem*{thm*}{Theorem}
\newtheorem{cor}[thm]{Corollary}
\newtheorem{lem}[thm]{Lemma}
\newtheorem{propo}[thm]{Proposition}

\newtheorem{conj}{Conjecture}
\theoremstyle{definition}

\theoremstyle{remark}
\newtheorem{rem}[thm]{Remark}
\theoremstyle{example}

\title{\thetitle}

\date{}

\title[Linear Boltzmann Equation with Soft Potentials]
{\textbf{On the Rate of Convergence to Equilibrium for the Linear Boltzmann Equation with Soft Potentials}}

\author{Jos\'e A. Ca\~nizo, Amit Einav and Bertrand Lods}

\address{Jos\'e A. Ca\~nizo, Departamento de Matem\'{a}tica
  Aplicada, Universidad de Granada, Av. Fuentenueva S/N, 18071
  Granada, Spain}
\email{canizo@ugr.es}
\thanks{JAC was supported by project MTM2014-52056-P, funded by the
  Spanish government and the European Regional Development Fund}
\address{Amit Einav, Institut f\"ur Analyisis und Scientific Computing,
Technische Universit\"at Wien, Austria}
\email{aeinav@asc.tuwien.ac.at}
\thanks{AE was {partially supported by EPSRC grant EP/L002302/1 and partially supported by the Austrian Science Fund (FWF) grant M 2104-N32}}

\address{Bertrand Lods, Departement of Economics and Statistics  \& Collegio
  Carlo Alberto, Universit\`{a} degli Studi di Torino,  Corso Unione
  Sovietica, 218/bis, 10134 Torino, Italy}
\email{bertrand.lods@unito.it}

\begin{document}

\maketitle
 
\begin{abstract}
  In this work we present several quantitative results of convergence
  to equilibrium for the linear Boltzmann operator with soft
  potentials under Grad's angular cut-off assumption. This is done by
  an adaptation of the famous entropy method and its variants,
  resulting in explicit algebraic, or even stretched exponential,
  rates of convergence to equilibrium under appropriate
  assumptions. The novelty in our approach is that it involves
  functional inequalities relating the entropy to its production rate,
  which have independent applications to equations with mixed linear
  and non-linear terms. We also briefly discuss some properties of the
  equation in the non-cut-off case and conjecture what we believe to
  be the right rate of convergence in that case.
\end{abstract}

\tableofcontents
\section{Introduction}\label{sec:intro}

This work is concerned with the asymptotic behaviour of the linear
homogeneous Boltzmann equation in the less explored case of soft
potential interactions, and with a cut-off assumption (the precise
definition of all the above will be given shortly). We are interested
in the application of entropy techniques to study the approach to
equilibrium in the relative entropy sense, and in the application of
entropy inequalities to estimate its rate. Our results complement a
previous work by two of the authors \cite{BCL}, where the case of hard
potentials was studied following the same techniques.

Our motivation comes partly from the study of the linear Boltzmann
equation itself, which is a basic model in kinetic theory describing
the collisional interaction of a set of particles with a thermal bath
at a fixed temperature. Interactions among the particles themselves
are neglected, and thus the equation is linear. Various versions of
the linear Boltzmann equation are used to model phenomena such as
neutron scattering \cite{MMK1, monta}, radiative transfer \cite{alo}
and cometary flows \cite{felln} (we refer to \cite[Chapter
XXI]{dautray} for a detailed presentation of the mathematical theory
of linear collisional kinetic equations), and appears in some
non-linear models as a background interaction term \cite{BiCaLo, CaLo,
  froh}. On the other hand, a technical motivation for our results is
that inequalities relating the logarithmic entropy to its production
rate are interesting by themselves, and are helpful in the study of
non-linear models involving a linear Boltzmann term. These inequalities
are intriguing and have been studied in \cite{BCL} in the case of hard
potentials; we intend to complete these ideas by looking at the case
of soft potentials. Our strategy of proof is close to that in
\cite{Lu} (which applies to the non-linear Boltzmann equation), and is
based on this type of inequalities.

The linear Boltzmann equation we consider here has been studied in
several previous works \cite{BCL, MMK, LMT, pettersson}. Its spectral
gap properties are understood since \cite{grad}, with constructive
estimates on the size of the spectral gap in $L^2(M^{-1})$ (where $M$
is the equilibrium) for hard potentials given in \cite{LMT}. Semigroup
techniques were used in \cite{MMK,MMK1} to obtain convergence to
equilibrium for all initial conditions in $L^1$, without explicit
rates. An important related equation is the \emph{linearised}
Boltzmann equation, which has been treated for example in
\cite{baranger,caflisch,levermore,Mou,strain}. Roughly speaking, the
spectral gap properties of both equations (linear and linearised) are
now understood in a variety of spaces. The difference in our present
approach is that it is based on functional inequalities for the
logarithmic entropy, which have their own interest and are more robust
when applied to models with mixed linear and non-linear terms
\cite{BiCaLo, CaLo}.

Similar questions for the non-linear space-homogeneous Boltzmann
equation have also been considered in the literature, and we refer to
\cite{DMV} for an overview and to \cite{Lu} for convergence results
with soft potentials. Mathematical questions are more involved in the
non-linear setting, and of course the picture becomes more complete in
the linear case. However, the question remains open regarding the
validity of some functional inequalities in the non-cutoff case; we
comment on this at the end of this introduction.

\subsection{The linear Boltzmann operator} 
In this work we will be interested in properties of the solution to
the following spatially homogeneous Boltzmann equation
\begin{equation}
  \begin{cases}
    \label{eq:BE}
    \partial_{t} f(t,v)&
    = \bm{L}_{\gamma}f(t,v) := \Q_{\gamma}(f(t,\cdot),\M)(v)
    \qquad t \geq 0
    \\
    f(0,v) &= f_{0}(v)
  \end{cases}
\end{equation}
where $\M$ is the Maxwellian with the same mass as $f_0$, and $\Q_{\gamma}(f,g)$ denotes the bilinear Boltzmann collision operator 
\begin{equation}\label{bolt}
\Q_{\gamma}(f,g)=
\int_{\R^{d}\times \mathbb{S}^{d-1}} B_{\gamma}(v-\vb, \cos \theta) \left(
f(\v')g(\vb')
-f(\v)g(\vb)\right)\d\vb \d\sigma\end{equation}
associated to a given interaction kernel of the form
\begin{equation}\label{eq:Bcs}
  B_{\gamma}(v-\vb,\cos\theta)=|v-\vb|^{\gamma}\,b(\cos\theta)
\end{equation}
with $\gamma \in (-d,0)$, and a given even nonnegative function on
$[-1,1]$, $b$, that satisfies
\begin{equation}\label{eq:grad}
  \|b\|_{1}=\int_{\sph}b(\cos\theta)\d
  \sigma=|\mathbb{S}^{d-2}|\int_{-1}^{1}b(s)\left(1-s^{2}\right)^{\frac{d-3}{2}}\d
  s < \infty.
\end{equation}
(the so-called Grad's angular cut-off assumption). For simplicity, we
will assume that $\|b\|_{1}=1$. The linear Boltzmann operator is then
defined by
$$\bm{L}_{\gamma}f = \Q_{\gamma}(f,\M)$$
In the above, $\v'$ and $\vb'$ are the pre-collisional velocities
which result, respectively, in the velocities $ \v $ and $\vb$ after
the elastic collision, expressed by the equation
  \begin{equation}
    \v'=\dfrac{\v +\vb}{2} + \dfrac{|v -\vb|}{2}\sigma,
    \qquad \vb'=\dfrac{\v +\vb}{2} - \dfrac{|v-\vb|}{2}\sigma\,,
\end{equation}
for a random unit vector $\sigma$. The deviation angle, which appears
in \eqref{eq:Bcs}, is defined by
$$\cos\theta=\frac{(\vb'-\v')\cdot (\vb-v)}{|\vb-v|^{2}}=
\frac{v-\vb}{|v-\vb|} \cdot \sigma.$$
The function $f$ considered in \eqref{bolt} is assumed to be a
non-negative function with unit mass. As such, the associated
normalised Maxwellian is given by
\begin{equation}\label{maxwe1}
\M(\v)=\frac{1}{(2\pi)^{d/2}}\exp
\left(-\dfrac{|\v|^2}{2}\right), \qquad \qquad \v
\in \R^d\,.
\end{equation}
Our study concerns itself with collision kernels of the form
\eqref{eq:Bcs} with $\gamma\in(-d,0)$. We will use the following, well
known, terminology:
\begin{enumerate}[{\bf 1)}]
\item If $\gamma > 0$ and $b$ satisfies \eqref{eq:grad}, we {are in
    the case of} \emph{hard potentials} with {angular} cut-off.
\item If $\gamma = 0$ and $b$ satisfies \eqref{eq:grad}, we {are in the case of} \emph{Maxwell interactions} with {angular} cut-off.
\item If $-d < \gamma < 0$ and $b$ satisfies \eqref{eq:grad}, we {are
    in the case of} \emph{soft potentials} with {angular} cut-off. 
\end{enumerate}

Our quantitative investigation of the rate of decay to equilibrium of
solutions to equation \eqref{eq:BE} uses the so-called entropy
method. The study of this method for the case of hard potentials has
been explored in \cite{BCL}, and the goal of this work is to extend
this study to the soft potentials case.

Before we present the main result of our work we recall in the next
section a few known facts about the linear Boltzmann equation.

\subsection{Known properties of the linear Boltzmann
  equation}
\label{sec:basic}

Basic results regarding equation \eqref{eq:BE} are its well-posedness
and the long time behavior of the Cauchy problem (see for instance
\cite{MMK}):
\begin{propo}
  Assume that $B(v-\vb,\sigma)=|v-\vb|^{\gamma}b(\cos\theta)$ where
  $\gamma \in (-d,0)$ and $b\::\:[-1,1] \to \R^{+}$ is {an even
    function that satisfies \eqref{eq:grad}}. {Then, the} operator
  $\bm{L}_{\gamma}$ is a bounded operator in $L^{1}(\R^{d})$ and, as
  such, generates a $C_{0}$-semigroup $(U(t))_{t\geq 0}$ of positive
  operators in $L^{1}(\R^{d}).$ Consequently, for any non-negative
  $f_{0} \in L^{1}(\R^{d})$ there exists a unique (mild) solution
  $f(t,\cdot)$ to \eqref{eq:BE} with $f(0,\cdot)=f_{0}$, given by
  $f(t)=U(t)f_{0}$. Moreover, $(U(t)_{t\geq 0}$ is a \emph{stochastic
    semigroup}, i.e.
$$\int_{\R^{d}}U(t)f_{0}(v)\d v =\int_{\R^{d}} f(t,v)\d v=\int_{\R^{d}}f_{0}(v)\d v \qquad \forall t \geq 0,$$
{and} for any $f_{0}\in L^{1}(\R^{d})$
$$\lim_{t \to \infty}\|U(t)f_{0}-\varrho_0 \M \|_{L^{1}(\R^{d})} =0$$
where $\varrho_0=\ds \int_{\R^d} f_{0}(v)\d v$.
\end{propo}
Notice that the above long-time behaviour of the solution to \eqref{eq:BE} does not require \emph{any additional assumption} on the initial datum. However, it does not provide any kind of rate of convergence for such general initial datum. {In fact, we will show} in the Appendix \ref{app:slow} that, without additional assumptions on the initial datum, the rate of convergence can be \emph{arbitrarily slow}. 

{From this point onwards, unless stated otherwise, we will assume that 
$$\varrho_{0}=\ds\int_{\R^{d}}f_{0}(v)\d v=1.$$} 

The first important observation in the study of the rate of
convergence to equilibrium is the fact that linear Boltzmann equation
\eqref{eq:BE} admits infinitely many Lyapunov functionals.

\begin{lem}
  \label{lem:lyapunov_functionals}
  Let $\mathbf{\Phi} \colon \R^{+}\to \R^{+}$ be a convex function and
  let $f(t,v)$ be non-negative solution to \eqref{eq:BE}. Then the
  functional
  $$H_{\mathbf{\Phi}}(f(t)|\M)=\int_{\R^d}\M(v)\mathbf{\Phi}\pa{\frac{f(t,v)}{\M(v)}}\d v$$
  is non-increasing. 
\end{lem} 
We refer to the Appendix \ref{app:BE} for a formal proof of that
property which is a general property of stochastic semigroups (see
\cite{rudnicki}). For the particular choice
$\mathbf{\Phi}(x)=x\log x -x+1$, one recovers the famous Boltzmann
relative entropy, which we will denote by $H(f|\M)$ and concludes the
$H-$Theorem:
$$\dfrac{\d}{\d t} H(f(t)|\M) \leq 0 \qquad \forall t \geq 0$$
for any solution $f(t,v)$ to \eqref{eq:BE} with unit mass {and initial
  data in an appropriate weighted space}. The rate at which the
relative entropy decreases is fundamental for the understanding of the
large time behaviour of $f(t,\cdot)$. Defining the \emph{entropy
  production} {as}:
$$\D_{\gamma}(f)=-\int_{\R^{d}}\,\bm{L}_{\gamma}(f)\log\pa{\frac{f(v)}{\M(v)}}\d v,$$
which is obtained by the minus of the formal derivative of the entropy
under the flow of the equation, the entropy method seeks to find a
general functional inequality that connects the entropy and the
entropy production. Such inequality is transformed into a differential
inequality along the flow of the equation, from which a concrete rate
of convergence to equilibrium can be obtained.

{The definition of $\D_\gamma$ can easily be extended to} any linear
Boltzmann operator $\bm{L}_{\gamma}$ with $\gamma>-d$. 
More generally, {we will denote the entropy production associated to a
  linear operator $\Q(f,\M)$ with collision kernel
  $B(v-v_{*},\sigma).$ by $\D_{B}(f)$, and an easy computation shows
  that}
\begin{equation}\label{eq:DB}
\D_{B}(f)=\frac{1}{2}\int_{\R^{d}\times\R^{d}\times\mathbb{S}^{d-1}}B(v-v_{*},\sigma)\M(v)\M(v_{*})\left(h(v')-h(v)\right)\log \frac{ h(v')}{ h(v)}\d v \d v_{*}\d\sigma\end{equation}
 where $h=\frac{f}{\M}.$ In particular, {as expected}, $\D_{B}(f) \geq 0$.
  
 The study of the entropy method is more developed for the Maxwellian
 and hard potentials case. In particular, we state the following
 theorem from \cite{BCL}, which will play an important role in our own
 study:

\begin{thm}\label{thm:entropy_method_hard_potentials}
  Consider a collision kernel $B(v-v_{*},\sigma)$ associated to
  Maxwell interactions
  $$B(v-v_{*},\sigma)=b(\cos\theta),$$
  where $b\::\:[-1,1]\to \R$ is an even function satisfying
  \eqref{eq:grad}. Then, there exists $\lambda_{0}>0$, depending only
  on $b$ such that
  \begin{equation}
    \label{eq:entropy_method_hard_potentials}
    \D_{0}(f) \geq \lambda_{0} H\pa{f|\M}.
  \end{equation}
  for any non-negative $f$ with unit mass {such that
    $$\int_{\R^d} \pa{1+\abs{v}^2} f(v) \abs{\log f(v)}\d v < \infty.$$}
\end{thm}

In general, we don't expect a linear inequality like
\eqref{eq:entropy_method_hard_potentials} relating the entropy
production to the relative entropy in the case of soft
potentials. Indeed, such an inequality would imply the existence of a
positive spectral gap in the space $L^{2}(\M^{-1})$ for the operator
$\bm{L}_{\gamma},$ which is known to be false (see \cite{caflisch} for
the linearised case and \cite{MMK} for the linear case). This is
  since the essential spectrum of $\bm{L}_{\gamma}$ can be shown to
contain a whole interval of the type $[-\nu_{0},0]$ {(see} Remark
\ref{rem:nospect} for more details and references on this topic).

The next type of functional inequality one may explore is the
following weaker inequality:
\begin{equation}\label{eq:Ddelta}
  \D_{\gamma}(f) \geq C_{\delta} H(f|\M)^{1+\delta}
\end{equation}
for some large class of probability densities $f$ and for some
explicit $\delta >0$ and $C_{\delta} >0$. In fact, to quantify the
long time behaviour of the linear Boltzmann equation, it is enough for
an inequality of the form \eqref{eq:Ddelta} to be valid along the flow
of solutions to \eqref{eq:BE}.

Next, we describe the main result of the present work.

\subsection{Main results}\label{sec:result}

Before stating our main results we will introduce some convenient
notation. Given a non-negative measurable function $f$, we denote the
$k$-th moment, and generalised $k$-th moment, of $f$ by
$$m_{k}(f)=\int_{\R^{d}}|v|^{k}f(v) \d v \qquad M_{k}(f)=\int_{\R^{d}}\langle v \rangle^{k}f(v)\d v \qquad \forall k \in \R$$
where $\langle v\rangle=\sqrt{1+|v|^{2}}$ for any $v \in \R^{d}.$
Moreover, given $s \in \R$ and $p \in (1,\infty)$, we set
$$M_{s,p}(f)=M_{s}(|f|^{p}),$$
and notice that $M_{0,p}(f)=\|f\|_{p}$. For a given $s \geq 0$ we
denote by
 $$\|f\|_{L^{1}_{s}}=M_s\pa{\abs{f}}$$
and
$$\|f\|_{L^{1}_{s}\log L}=\int_{\R^{d}}\langle v\rangle^{s} |f(v)|\,|\log |f(v)\,|| \d v,$$
and define the function spaces
\begin{gather*}
  L^{1}_{s} =
  L^{1}_{s}(\R^{d}) =
  \left\{ f \colon \R^{d}\to \R \mid
    \text{$f$ measurable and } \|f\|_{L^{1}_{s}} <\infty
  \right\}
  \\
  L^{1}_{s}\log L =
  \left\{f \colon \R^{d}\to \R \mid
    \text{$f$ is measurable and }
    \|f\|_{L^{1}_{s}\log L} <\infty
  \right\}.
\end{gather*}
Even if $\|\cdot\|_{L^{1}_{s}\log L}$ is not a norm, this notation is
commonly seen in the literature.

\medskip We are now ready to state our first main result.
 
\begin{thm}
  \label{thm:rate_of_convergence}
  Take $p > 1$ and $-d < \gamma < 0$, and let
  $f_0\in L^1_s\pa{\R^d} \cap L^p(\R^d)$ be a non-negative function
  with unit mass, for $s\geq s_{p,d,\gamma}$, where
  $s_{p,d,\gamma} > 2 + |\gamma|$ is an explicit constant that depends
  only on $p,d$ and $\gamma$. Let $f = f(t)$ be the solution to
  equation \eqref{eq:BE} with a bounded angular kernel $b$. Then for any
  $$\sigma < -1+\frac{s-2}{\abs{\gamma}}$$
  there exists a uniform constant $C_{0}>0$ depending only on
  $d,\gamma,p,s,\sigma,\norm{f_0}_{L^1_s},\norm{f_0}_p$ and
  $H\pa{f_0|\M}$ such that
  \begin{equation}
    \label{eq:rate_of_convergence}
    H\pa{f(t)|\M} \leq C_{0}\pa{1+t}^{-\sigma},
  \end{equation}
  for all $t \geq 0$.
\end{thm}
 {The strategy of the proof is to obtain the inequality
 $$\D_\gamma(f) \geq \D_0(f)^{\frac{\mu-\gamma}{\mu}}\D_{\mu}(f)^{\frac{\gamma}{\mu}}$$
 for some $\gamma < 0 < \mu$ by means of interpolation estimates and deduce from it the inequality 
 \begin{equation}\label{eq:entr}
\D_{\gamma}(f) \geq C(f)^{\frac{\gamma}{\mu}}\,H(f|\M)^{1-\frac{\gamma}{\mu}}\end{equation}
where $C(f)$ is an explicit functional involving norms of $f$ in appropriate $L^1_{\kappa_1}$ and $L^1_{\kappa_2} \log L$ spaces, for a suitable $\kappa_1,\kappa_2$. To use this inequality to deduce Theorem \ref{thm:rate_of_convergence} one needs to control $C(f)$ along the flow of the equation. This is achieved by obtaining the following:}
\begin{enumerate}[(i)]
\item {Explicit time dependent upper} bound on the moments $m_{s}(f(t))$ of the solutions.
\item {Explicit time dependent upper} bound on the $L^{p}$-norms $\|f(t)\|_{p}$.
\item {Pointwise} Gaussian lower bounds for the solutions $f(t,v)$.
\end{enumerate}
The methods we use to obtain the above estimations are inspired by the
works \cite{Lu} and \cite{ToVi}, that deal with a similar problem
related to the non-linear Boltzmann equation. One consequence of
Theorem \ref{thm:rate_of_convergence} is that one can immediately use
interpolation in order to show uniform-in-time bounds of moments and
$L^p$ norms of the solution, assuming that a large enough moment and
$L^q$ norm is initially finite; see Theorems \ref{thm:moments-bounded}
and \ref{thm:Lp-bounded} for a precise statement.

\medskip
Our second main result concerns the decay of the solution to \eqref{eq:BE} for a more restrictive class of initial datum satisfying {a} strong Gaussian {control of the form}
$$\int_{\R^d}\M(v)^{1-p}f_0(v)^p\d v<\infty$$
for some $p >1.$ In this case, one can obtain a better rate of
decay---one of the form of a stretched exponential:

\begin{thm}\label{thm:rate_of_convergence_exp}
  Let $f_0\in L^1\pa{\R^d}$ be a non-negative function such that
  $f_0\in L^p\pa{\R^d}$ for some $p>1$ {and where $d\geq 2$}. Assume
  that the angular kernel, $b$, is bounded and satisfies
  \begin{equation}\label{eq:blower}
    b(x)\geq b_0\pa{1-x^2}^{\frac{\nu}{2}},
  \end{equation}
  for some $b_0>0$, $0\leq \nu \leq 1$. Then, if 
  $$H_p(f_0)=\int_{\R^d}\M(v)^{1-p}f_0(v)^p\d v<\infty$$
  we have that for any $t_0>0$ there exist two uniform constants
  $C^{(1)}_{t_0},C^{(2)}_{t_0}>0$ depending only on
  $d,\gamma,p,b_0,\nu,t_0$ and $H_p(f_0)$ such that any non-negative
  solution to {\eqref{eq:BE}} with initial data $f_0$, $f(t)$,
  satisfies
  \begin{equation}
    \label{eq:almost_cerc_exp}
    \D_\gamma (f(t)) \geq \frac{C^{(1)}_{t_0}H\pa{f(t)|\M}}{\abs{\log\pa{C^{(2)}_{t_0}H\pa{f(t)|\M}}}^{\frac{\abs{\gamma}}{2}}},\quad \forall t \geq t_0. 
  \end{equation}
  As a consequence we can find appropriate constants
  $C_{1}, \lambda_{1} >0$ depending on $d,\gamma,p,b_0$ such that
  \begin{equation}
    \label{eq:stretched_exp_decay}
    H\pa{f(t)|\M} \leq C_{1}\exp\pa{-\lambda_{1}t^{\frac{2}{2+\abs{\gamma}}}}.
  \end{equation}
\end{thm}
 
The above decay rate is similar to that obtained for the linearised
Boltzmann equation in \cite{caflisch}, yet with a less restrictive
condition on the initial datum. Indeed, the condition in
\cite{caflisch} involves a \emph{pointwise} Gaussian decay of the type
$$\sup_{v\in\R^{d}} \exp\pa{a |v|^{2}}|f(v)| < \infty$$
for some $a \in (0,1/4).$ 

The proof of Theorem \ref{thm:rate_of_convergence_exp} uses a suitable
improvement of the interpolation inequality between $\D_\gamma$ and
$\D_0$, which involves now the entropy production associated to some
(non-physical) interaction kernel of the form
$$B(v-\vb,\sigma)=\exp(a|v-\vb|)\,b(\cos\theta).$$
This is reminiscent of a similar approach used in the study of the
entropy production associated to the Becker-D\"oring equation {in}
\cite{ACL}. An additional ingredient of the proof is the instantaneous
generation of a Maxwellian lower bound to the solutions of
\eqref{eq:BE}, which is the reason why assumption \eqref{eq:blower} is
needed.

\medskip
We also consider the non-cutoff case briefly in Section
\ref{sec:nonc}. If one assumes that
\begin{equation}
  \label{eq:bnoncut-intro}
  c_{0}|\theta|^{-(d-1)-\nu} \leq b(\cos\theta)
  \leq c_{1}|\theta|^{-(d-1)-\nu}, \qquad \nu \in (0,2)
\end{equation}
for certain positive constants $c_{1} \geq c_{0} >0$ then the cutoff
assumption \eqref{eq:grad} is not satisfied. The spectral gap
properties of the \emph{linearised} Boltzmann equation are
well-understood also in this case \cite{strain,gress}, and by
following the technique in \cite{strain} we show a analogous result
for the linear Boltzmann equation: if $\gamma + \nu > 0$ the operator
$\bm{L}_\gamma$ has a spectral gap in the space $L^2(\M^{-1})$ (see
Proposition \ref{propo:spectral}). Since we are interested in
inequalities involving the logarithmic entropy, we may wonder whether
a similar linear inequality holds true for the entropy production
$$\D(f)=-\int_{\R^{d}}\bm{L}f \log \left(\frac{f}{\M}\right)\d v.$$
While we have not been able to prove this, we conjecture that it is
indeed the case. More precisely:
\begin{conj}
  For a non cut-off collision kernel
  $B(v-\vb,\sigma)=|v-\vb|^{\gamma}b(\cos\theta)$ with
  $\gamma \in (-d,0)$ and $b(\cdot)$ satisfying \eqref{eq:bnoncut-intro}
  such that
  $$\gamma+\nu >0$$
  there exists $\lambda_{\gamma,b} >0$ such that
  \begin{equation}
    \label{eq:DfLin}
    \D(f) \geq \lambda_{\gamma,b}H(f|\M)\end{equation}
  for all $f \geq 0$ with unit mass.
\end{conj}

A linear inequality like \eqref{eq:DfLin} is usually refer to as a
\emph{modified Logarithmic Sobolev inequality} and is known to be
equivalent to the exponential decay of $H(f(t)|\M)$ along the flow of
solutions to the Boltzmann equation
\begin{equation}
  \label{eq:BEnon}
  \dfrac{\d}{\d t}f(t,v)=\bm{L}f(t,v),
  \qquad f(0,\cdot)=f_{0} \in L^{1}_{2}\log L
\end{equation}
(see for instance \cite{tetali}). Such a modified Logarithmic Sobolev
inequality would imply the spectral gap inequality
\eqref{eq:spectralgap} with
$$\lambda \geq \frac{\lambda_{\gamma,b}}{2}$$
(but is not equivalent to it).

\subsection{Organization of the paper}

{The structure of the paper is as follows: Section
  \ref{sec:entropy_inequality}} is dedicated to the main
entropy-entropy production inequality of the type \eqref{eq:entr} {and
  to the investigation of points \textit{(i)--(iii)}, leading to the
  proof of our first main result} in Section
\ref{sec:rate_of_convergence}. In Section
\ref{sec:maxwellian_lower_bounds} we show the creation of pointwise
Maxwellian lower bounds {under certain restrictions on the angular
  kernel. This will not only give an alternative to point
  \textit{(iii)} (which will not improve the rate of convergence for
  Theorem \ref{thm:rate_of_convergence}), but will be crucial in the
  proof of Theorem \ref{thm:rate_of_convergence_exp}, which we will
  give in Section \ref{sec:exp}}. In Section \ref{sec:nonc}, we
discuss the case of the linear Boltzmann equation with soft potential
\emph{without the cut-off} assumption
{and show the existence of a spectral gap for a certain range of the
  parameters. This is done by an adaptation of similar results from
  \cite{strain}}. {The last pages of the paper are dedicated to
  several Appendices that provide additional details that we felt
  would hinder the flow of the main work.}

\section{The Entropy Inequality and Technical
  Estimates}\label{sec:entropy_inequality}

The goal of this section is to find an appropriate entropy-entropy
production inequality associated to $\bm{L}_{\gamma}$, from which we
will be able to obtain a quantitative estimation on the rate of
convergence to equilibrium.

In order to achieve this we start by rewriting the operator
$\bm{L}_{\gamma}$ as the sum of a gain and a loss operators. Due to
the cut-off assumption \eqref{eq:grad} the operator $\bm{L}_{\gamma}$
can be decomposed in the following way:
$$\bm{L}_{\gamma}f(v)=\bm{K}_\gamma f(v)-\Sigma_\gamma(v)f(v),$$
where
\begin{equation}\label{eq:def_of_K}
\bm{K}_\gamma f(v)=\int_{\R^d\times \sph} \abs{v-v_\ast}^\gamma b\pa{\cos{\theta}}f\pa{v^\prime}\M\pa{v_\ast^\prime}\d\vb\d\sigma
\end{equation} 
and the collision frequency $\Sigma_{\gamma}$ is given by
\begin{equation}\label{eq:def_of_Sigma}
\Sigma_\gamma(v)=\int_{\R^d\times \sph} b(\cos\theta)\abs{v-v_\ast}^\gamma\M\pa{v_\ast}\d\vb\d\sigma=\int_{\R^{d}} \abs{v-v_\ast}^\gamma\M\pa{v_\ast}\d\vb.\end{equation}
The loss operator, of a simpler nature, satisfies the following (see \cite{caflisch} or \cite[Lemma 6.1]{Lu} for a detailed proof):
\begin{lem}\label{lem:loss_operator}
For any {$\gamma\in \R$} there exists explicit constants $C_1,C_2>0$, depending only on $\gamma,d$ and $\norm{b}_1$ such that
\begin{equation}\label{eq:loss_operator}
C_1\pa{1+\abs{v}}^\gamma \leq \Sigma_\gamma(v) \leq C_2\pa{1+\abs{v}}^\gamma.
\end{equation}
\end{lem}
\begin{rem}\label{rem:nospect}
{From the above estimate one can easily infer that} the range of the mapping $-\Sigma_{\gamma}$ is given by $[-\nu_{0},0)$ for $\nu_{0}=\inf_{v\in \R^{d}}\Sigma_{\gamma}(v)$. Using the fact that $\bm{K}_{\gamma}$ is a compact operator in the space $L^{2}(\M^{-1})$ (see for instance \cite{caflisch} or \cite{levermore,guo} for the proof {in the linearised setting}), one deduces that the essential spectrum of $\bm{L}_{\gamma}$ in that space contains  $[-\nu_{0},0)$. In particular, $\bm{L}_{\gamma}$ does not exhibit a spectral gap in that space. 
\end{rem}
We are now ready to state our main entropy inequality.
\begin{thm}\label{thm:entropy_inequality}
Let $\gamma\in(-d,0)$ and let $f\in L^{1}_{\mu}\log L\pa{\R^{d}} \cap L^{1}_{\mu+2}(\R^{d})$ {for some $\mu>0$, be a non-negative function} with unit mass. {Then}
\begin{equation}\label{eq:DgD0Dmu}\D_\gamma(f) \geq \D_0(f)^{\frac{\mu-\gamma}{\mu}}\D_{\mu}(f)^{\frac{\gamma}{\mu}}\geq \lambda_{0}^{1-\frac{\gamma}{\mu}}\D_{\mu}(f)^{\frac{\gamma}{\mu}}\,H\pa{f|\M}^{1-\frac{\gamma}{\mu}},\end{equation}
and it also holds that
\begin{multline}
  \label{eq:entropy_inequality}
  \D_{\mu} (f) \geq C_{s,d}\Bigg(\int_{\R^d}\pa{1+\abs{v}}^{\mu} f(v)\log f(v)\d v \\
  +\int_{\R^d}\pa{1+\abs{v}}^{\mu+2} f(v) \d v-\int_{\R^d}\bm{K}_{\mu}(f)(v)\log f(v) \d v \Bigg)
\end{multline}
where $C_{\mu,d}$ is a universal constant that depends only on $\mu$
and $d$, and $\lambda_{0}$ is the positive parameter (depending on
$b$) appearing in Theorem \ref{thm:entropy_method_hard_potentials} for
Maxwell molecules.
\end{thm}

\begin{proof}
  Recall that, for all $\alpha >-d$
$$\D_{\alpha}(f)=\frac{1}{2}\int_{\R^{d}\times\R^{d}\times\mathbb{S}^{d-1}}|v-v_{*}|^{\alpha}b(\cos\theta)\M(v)\M(v_{*})\left(h(v')-h(v)\right)\log \frac{ h(v')}{ h(v)}\d v \d v_{*}\d\sigma.$$
Introducing the measure $\d\nu(v,\vb,\sigma)=\frac{1}{2}b(\cos\theta)\M(v)\M(v_{*})\left(h(v')-h(v)\right)\log \frac{ h(v')}{ h(v)}\d v \d v_{*}\d\sigma$ on $\R^{d}\times\R^{d}\times \sph$ one has
$$\D_{\alpha}(f)=\int_{\R^{d}\times\R^{d}\times\mathbb{S}^{d-1}}|v-v_{*}|^{\alpha}\d\nu(v,v_{*},\sigma)$$
and, using H\"older's inequality {on
$$\D_0(f) = \int_{\R^d} |v-v_{*}|^{\frac{\mu\gamma}{\mu-\gamma}}|v-v_{*}|^{\frac{-\mu\gamma}{\mu-\gamma}}\d\nu(v,\vb,\sigma)$$} 
with $p=\frac{\mu-\gamma}{\mu}$, $q=-\frac{\mu-\gamma}{\gamma}$  we get 
\begin{equation}\label{eq:entropy_inequality_proof_I}
\D_\gamma(f) \geq \D_0(f)^{\frac{\mu-\gamma}{\mu}}\D_{\mu}(f)^{\frac{\gamma}{\mu}}.
\end{equation}
Next, as $\D_{\mu}(f)=-\ds\int_{\R^d}\Lin_{\mu}(f)\log (f/\M)\d v$ we have that
\begin{multline*}
\D_{\mu}(f) \leq \int_{\R^d}\Sigma_{\mu}(v)f(v)\log f(v)\d v + \int_{\R^d}\bm{K}_{\mu}(f)(v)\log \M(v)\d v\\
-\int_{\R^d}\Sigma_{\mu}(v)(f)\log\M(v)\d v -\int_{\R^d}\bm{K}_{\mu}(f)(v)\log f(v)\d v.\end{multline*}
Since $\log\pa{\M(v)}=-\frac{d}{2}\log \pa{2\pi}-\abs{v}^2/2<0$ and $\bm{K}_{\mu}(f)(v) \geq 0$ when $f$ is non-negative we conclude that 
$\ds \int_{\R^d}\bm{K}_{\mu}(f)(v)\log \M(v)\d v \leq 0.$ Moreover, 
using {Lemma \ref{lem:loss_operator}} we find $C_1,C_2>0$, depending only on $d$ and $\mu$ such that
\begin{multline*}
\D_{\mu}(f) \leq C_2 \int_{\R^d}\pa{1+\abs{v}}^{\mu} f(v)\log f(v)\d v + \frac{d}{2}\log\pa{2\pi}C_2 \int_{\R^d}\pa{1+\abs{v}}^{\mu} f(v) \d v\\
+\frac{C_2}{2}\int_{\R^d}\pa{1+\abs{v}}^{\mu+2} f(v)\d v-\int_{\R^d}\bm{K}_{\mu}(f)(v)\log\pa{f(v)}\d v.\end{multline*}
The above, together with \eqref{eq:entropy_inequality_proof_I} and Theorem \ref{thm:entropy_method_hard_potentials} complete the proof.
\end{proof}
The above Theorem is the reason for us to investigate the evolution of
moments and $L^p$ norms of $f$, as well as pointwise lower bounds of
$f$. These are the topics of the following subsections. From now on, we shall always assume that $\gamma \in (-d,0).$

\subsection{The Evolution of Moments}
\label{sec:moments}

The study of moments and their time evolution is fundamental in many
kinetic equations (and other PDEs where ``energy methods'' are
applicable). In the case of the Boltzmann equation, the study of
creation and propagation of moments for soft and hard potentials with
angular cut off is radically different. The linear Boltzmann equation we study here
exhibits properties that reflect a similar moment growth as its
non-linear counterpart.

We recall that notations for the moments of $f$ have been introduced
in {subsection \ref{sec:result}}. {To simplify writing, we will denote
  by
$$m_{s}(t)=m_{s}(f(t,\cdot)), \qquad  M_{s}(t)=M_{s}(f(t,\cdot)) \qquad \forall t \geq 0, \quad s \in \R.$$ }
{For a given $s\in\R$ we define the function
\begin{equation}\label{eq:def_of_bold(w)s}
\bm{w}_s(v)=\int_{\R^d}\abs{v-v_\ast}^{\gamma}\abs{v_\ast}^s \M(v_\ast)\d\vb,
\end{equation}
which will play an important role in the sequel and which satisfies the following estimate (similar in nature to Lemma \ref{lem:loss_operator})} 
\begin{lem}
  \label{lem:boundedness_of_bold(m)}
  For any $s\geq 0$ and $\gamma > -d$, $\bm{w}_{s}$ is a bounded
  function. That is,
  \begin{equation}
    \label{eq:boundedness_of_bold(m)}
    \sup_{v\in\R^d}\bm{w}_s(v) :=
    \norm{\bm{w}_{s}}_{\infty} <\infty.
  \end{equation}
\end{lem}
 The main theorem we prove in this section is the following:
 \begin{thm}
   \label{thm:evolution_of_moments}
   Let $f_0\in L_s^1\pa{\R^d}$ for {$s=2$ or
     $s>2\max\pa{\abs{\gamma},1}$} such that $f_0$ has unit mass and
   let $f(t)$ be the unique solution for \eqref{eq:BE}. If
   $b \in L^{\infty}(\sph) $ and $\|b\|_{1}=1$ there exists a constant
   $C_s$, depending only on $s,\gamma,d$, the collision kernel and
   $\norm{f_0}_{L^1_s\pa{\R^d}}$ such that
   \begin{equation}
     \label{eq:evolution_of_moments}
     m_s(t) \leq C_s (1+t),\quad\quad \forall t\geq 0.
   \end{equation}
 \end{thm}
 In order {to prove} the above theorem we will need to use the
 so-called Povzner's Lemma (see \cite{MiWe}). The version we present
 here can be found in greater generality in
 \cite{BrEi}
 \begin{lem}
   \label{lem:Povzner}
   Assume that the angular kernel $b(\cdot)$ is a bounded function and
   let $s>2$. Setting
   $$I_s(v,v_\ast)=\int_{\sph}b\pa{\cos\theta}\pa{\abs{v^\prime}^s+\abs{v^\prime_\ast}^s-\abs{v}^s-\abs{v_\ast}^s}\d\sigma, \qquad v,\vb \in \R^{d}$$ 
   we have that
   \begin{equation}
     \label{eq:PovznberI}
     I_{s}(v,v_\ast) \leq C^{(1)}_s \abs{v}^{\frac{s}{2}}\abs{v_\ast}^{\frac{s}{2}} - C^{(2)}_s\pa{\abs{v}^{s}+\abs{v_\ast}^{s}}\pa{1-\bm{1}_{\frac{\abs{v}}{2} \leq \abs{v_\ast}\leq 2\abs{v}}(v,v_\ast)},
   \end{equation}
   where $C^{(1)}_s,C_s^{(2)}$ are positive constants that depend only
   on $d,s$ and the angular kernel, and where $\bm{1}_A$ is the
   indicator function of the set $A$.
 \end{lem}

 The reason we use this version of Povzner's Lemma rather than others
 (such as the one in \cite{gpv}) is due to the fact that it gives a
 minimal order of growth in terms of $|v|$. As any order of growth in
 $|\vb|$ is absorbed by the Maxwellian, the result obtained with this
 version of the lemma is optimal in our setting.

 \begin{proof}[Proof of Theorem \ref{thm:evolution_of_moments}]
 We start with considering the case $s=2$ as it doesn't require the boundedness of $b$, due to the special geometry in this case. Using the natural pre-post collision change of variables we find that
 $$\frac{\d}{\d t}m_2(t) = \int_{\R^d\times \R^d \times \sph}\abs{v-v_\ast}^\gamma b\pa{\cos\theta}f(t,v)\M(v_\ast)\pa{\abs{v^\prime}^2-\abs{v}^2}\d v\d\vb\d\sigma.$$
 As 
 $$\abs{v^\prime}^2-\abs{v}^2 = \abs{v_\ast}^2 - \abs{v_\ast^\prime}^2$$
and $\norm{b}_{1}=1$, we find that
 $$\frac{\d}{\d t}m_2(t)  \leq \int_{\R^d}f(t,v)\bm{w}_2(v)\d v \leq  \|\bm{w}_{2}\|_{\infty},$$
 where we have used the mass conservation property of the equation. Thus
 \begin{equation}\label{eq:evolution_of_second_moment}
 m_2(t) \leq \max\pa{m_2(0),\|\bm{w}_{2}\|_{\infty}}\pa{1+t}.
 \end{equation}
 Next, we consider $s>2$. Similar to the above we find that
 \begin{equation*}\begin{split}
 \frac{\d}{\d t}m_s(t) &= \int_{\R^d\times \R^d \times \sph}\abs{v-v_\ast}^\gamma b\pa{\cos\theta}f(t,v)\M(v_\ast)\pa{\abs{v^\prime}^s-\abs{v}^s}\d v\d\vb\d\sigma\\
 &=\int_{\R^d\times \R^d \times \sph}\abs{v-v_\ast}^\gamma b\pa{\cos\theta}f(t,v)\M(v_\ast)\pa{\abs{v^\prime}^s+\abs{v^\prime_\ast}^s-\abs{v}^s-\abs{v_\ast}^s}\d v\d\vb\d\sigma\\
 &\phantom{+++} +\int_{\R^d\times \R^d \times \sph}\abs{v-v_\ast}^\gamma b\pa{\cos\theta}f(t,v)\M(v_\ast)\pa{\abs{v_\ast}^s-\abs{v_\ast^\prime}^s}\d v\d\vb\d\sigma\\
&\leq \int_{\R^d\times \R^{d}}\abs{v-v_\ast}^\gamma f(t,v)\M(v_\ast)I_s(v,v_\ast)\d v\d\vb +  \|\bm{w}_{s}\|_{\infty} \end{split}\end{equation*}
Using Lemma \ref{lem:Povzner}, together with the fact that 
 $$\pa{\abs{v}^{s}+\abs{v_\ast}^{s}}\bm{1}_{\frac{\abs{v}}{2} \leq \abs{v_\ast}\leq 2\abs{v}}(v,v_\ast) \leq \pa{2^s+1}\abs{v_\ast}^s,$$
 we conclude that there exist appropriate universal constants $C_{i}$ $(i=1,2,3)$ that depends only on $\gamma,d$ and the angular kernel such that
 \begin{equation*}\begin{split}
 \frac{\d}{\d t}m_s(t) &\leq C_s^{(1)} \|\bm{w}_{\frac{s}{2}}\|_{\infty}m_{\frac{s}{2}}(t)-C_s^{(2)}\int_{\R^d\times \R^d} \abs{v-v_\ast}^\gamma \abs{v}^s f(t,v) \M(v_\ast)\d v\d\vb\\
&\phantom{++++} +\pa{C_s^{(2)}\pa{2^s+1}+1}\|\bm{w}_{s}\|_{\infty}\\
&\leq C_1 m_{\frac{s}{2}}(t) - C_2\int_{\R^d}\Sigma_\gamma(v)\abs{v}^s f(t,v)\d v+C_3. \end{split}\end{equation*}
Using Lemma \ref{lem:loss_operator} we find that, using abusive notations for the constants,

 \begin{equation*}\begin{split} \frac{\d}{\d t}m_s(t) &\leq C_1m_{\frac{s}{2}}(t)-C_2\int_{\R^d} \pa{1+\abs{v}}^\gamma\abs{v}^s f(t,v)\d v+C_3\\
&= C_1m_{\frac{s}{2}}(t)-C_2\int_{\R^d} \avg{v}^\gamma\pa{1+\abs{v}^s} f(t,v)\d v+C_2\int_{\R^d} \avg{v}^\gamma f(t,v)\d v+C_3\end{split}\end{equation*}
i.e.
$$\frac{\d}{\d t}m_s(t) \leq C_1m_{\frac{s}{2}}(t)-C_2M_{s+\gamma}(t)+C_3.$$
Since $s>2\abs{\gamma}$ we see that $s+\gamma > s/2$ and as such
$$m_{\frac{s}{2}}(t) \leq M_{\frac{s}{2}}(t) \leq M_{s+\gamma}(t)^{\frac{s}{2(s+\gamma)}}.$$
Thus, in our settings, 
$$\frac{\d}{\d t}m_s(t) \leq M_{s+\gamma}(t)^{-\frac{s+2\gamma}{2(s+\gamma)}}\pa{C_1 M_{s+\gamma}(t)-C_2M_{s+\gamma}(t)^{1+\frac{s+2\gamma}{2(s+\gamma)}} }+C_3. $$
Since for any $\delta>0$ the exists a constant $C(a,b,\delta)>0$ such that
$$\sup_{x>0}\pa{ax-bx^{1+\delta}} \leq C(a,b,\delta),$$
and since $s+2\gamma>0$ and $M_s(f)\geq m_0(f)$, we conclude that there exists appropriate constants such that
$$\frac{d}{dt}m_s(t) \leq C_1 M_{s+\gamma}(t)^{-\frac{s+2\gamma}{2(s+\gamma)}}+C_3 \leq C_s, $$
completing the proof.
 \end{proof}

 Theorem \ref{thm:evolution_of_moments} gives us the tools to improve any growth estimation of a given moment, as long as the initial data has higher moments.
 \begin{cor}\label{thm:improved_evolution_of_moments}
Let {$s_1=2$ or $s_1> 2\max\pa{1,\abs{\gamma}}$, and {let} $s_2 \geq s_1$. Then, if $f_0\in L_{s_2}^1\pa{\R^d}$ with a unit mass, and $b \in L^{\infty}(\sph)$ such that $\|b\|_{1}=1$, we have that
 \begin{equation}\label{eq:evolution_of_moments}
 m_{s_1}(t) \leq C_{s_2} \pa{1+t}^{\frac{s_1}{s_2}},\quad\quad \forall t\geq 0.
 \end{equation}
for some constant $C_{s_2}$, depending only on $s_2,\gamma,d$, the collision kernel $b(\cdot)$ and $\norm{f_0}_{L^1_{s_2}\pa{\R^d}}$.}
 \end{cor}
 \begin{proof}
 This  follows from simple interpolation.\end{proof}
 
\subsection{$L^p$ Estimates}

The goal of this subsection is to show the propagation of $L^p$
bounds---as long as one has enough moments. The approach we present
here follows that of \cite{ToVi}. The main result we will show is:

\begin{thm}
  \label{thm:L_p_bounds}
  Let $f_0\in L^1\pa{\R^{d}}\cap L^p\pa{\R^d}$ for some $p>1$ and let
  $f = f(t,v)$ be a non-negative solution to \eqref{eq:BE}. 
  There exists a constant $C_{p,d,\gamma}$ depending only on $p,d$ and $\gamma$ and $\bm{r} >1$  such that
  \begin{equation}
    \label{eq:evolution_of_L_p}
    \norm{f(t,\cdot)}^p_{p} \leq \norm{f_0}_p^p+ C_{p,d,\gamma} \int_0^ t M^p_{\bm{r}}(\tau)\d\tau.
  \end{equation}
\end{thm}

\begin{rem}
  \label{rem:advantage_of_p=2}
Under the assumptions of the above Theorem, setting
$$\eta_{0}:=\min\pa{p-1, \frac{\pa{p-1}^2}{\pa{2-p}^{+}},\frac{p^2}{\sqrt{\frac{1}{\pa{p-1}^2}+p^2}+\frac{1}{p-1}},\frac{p^2}{\sqrt{\frac{1}{\pa{1+\frac{\gamma}{d}}^2}+p^2}+\frac{1}{1+\frac{\gamma}{d}}}}$$
  where $a^{+}=\max(a,0)$, we can deduce from the proof that, for any $0 < \eta < \eta_{0}$, one can chose
 \begin{equation}\label{eq:def_of_S}
    \bm{r}=\frac{|\gamma|}{\eta}\max\pa{p-1-\eta,(p+\eta)(p-2)+1}.
  \end{equation} 
 
  It is easy to see that $p-1-\eta<(p+\eta)(p-2)+1$ if and only if
  $\eta>2-p$. This means that whenever $p\geq 2$ we have
$$\bm{r}=-\frac{\gamma}{\eta}\left((p+\eta)(p-2)+1\right).$$
{In addition, we notice that if} $p \to 1^{+}$, one has $\eta_{0} \simeq \frac{1}{2}(p-1)^{2}$. {Choosing} $\eta \simeq \eta_{0}$ we get 
{$$\bm{r}\simeq_{p \to 1^{+}} -\frac{2\gamma}{(p-1)^{2}} \left(p-1-\frac{(p-1)^{2}}{2}\right) \simeq -\frac{2\gamma}{p-1}.$$}
\end{rem}

\begin{rem}
  Before we set the stage for the proof of Theorem
  \ref{thm:L_p_bounds}, we would like to note an important difference
  between the study of the linear Boltzmann equation and the fully
  non-linear equation in this setting. The work \cite{ToVi} deals with
  collision kernels of the form
  $B(v-\vb,\sigma)=\Phi(|v-\vb|)b(\cos\theta)$ where
  $$(1+r)^{\gamma} \leq \Phi(r) \leq c_{2}(1+r)^{\gamma}$$ 
  for some $c_{1},c_{2} >0.$ Our version of the $L^p$ bound
  propagation, however, deals directly with $\phi(r)=r^\gamma$. The
  reason we are able to do that is the presence of the Maxwellian in
  the collision operator, acting as a mollifier to the singularity.
\end{rem}

The proof of this Theorem relies on the following integrability
property of $\bm{K}_{\gamma}$ that can be found in \cite{AlCaGa}:
\begin{thm}
  \label{thm:Alonso_Carneiro_Gamba}
  Let $\gamma\in (-d,0)$ and $1<r,q,\ell<\infty$ with
  \begin{equation}
    \label{eq:Alonso_Carneiro_Gamba_coef}
    \frac{1}{q}+\frac{1}{\ell}=1+\frac{\gamma}{d}+\frac{1}{r}.
  \end{equation}
  Then the gain operator $\bm{K}_\gamma$ satisfies
  \begin{equation}
    \label{eq:Alonso_Carneiro_Gamba}
    \norm{\bm{K}_\gamma(f)}_r \leq C \norm{f}_q \norm{\M}_\ell
  \end{equation}
  where $C=C(r,q,\ell,\gamma,d) >0$ is a uniform constant that depends
  only on $r,q,\ell,\gamma$ and $d$.
\end{thm}

\begin{proof}[Proof of Theorem \ref{thm:L_p_bounds}]
  We recall the notation $M_{p,s}(f)=M_{s}\pa{\abs{f}^p}$ and, as
  before, set $M_{p,s}(t)=M_{p,s}(f(t,\cdot))$ for any $t \geq 0$,
  where $f(t,v)$ is the unique solution to \eqref{eq:BE}. We have
  that:
  $$\frac{\d}{\d t}\frac{1}{p}\norm{f(t)}_p^p= \int_{\R^d}\bm{K_\gamma}(f(t))(v) f^{p-1}(t,v)\d v - \int_{\R^d}\Sigma_{\gamma}(v)f^{p}(t,v)\d v.$$
  Using Lemma \ref{lem:loss_operator}, the equivalence of
  $\pa{1+\abs{v}}^\gamma$ and $\avg{v}^\gamma$, and taking
  $r,r^\prime>1$ to be H\"older conjugates (with $r$ to be determined
  shortly) we find that
  $$\frac{\d}{\d t}\frac{1}{p}\norm{f(t)}_p^p \leq \norm{\bm{K}_\gamma (f(t))}_{r}\norm{f(t)}_{r^\prime(p-1)}^{p-1}-c_\gamma M_{p,\gamma}(t),$$
  for some appropriate constant $c_\gamma$. Using Theorem
  \ref{thm:Alonso_Carneiro_Gamba} with $q,\ell>1$ to be fixed later,
  we conclude that
  \begin{equation}
    \label{eq:L_p_estimates_0}
    \frac{\d}{\d t}\frac{1}{p}\norm{f(t)}_p^p \leq C\norm{ f(t)}_{q}\norm{f(t)}_{r^\prime(p-1)}^{p-1}-c_\gamma M_{p,\gamma}(t).
  \end{equation}
  for some uniform constant $C$, depending only on $r,q,\ell,\gamma$
  and $d$ satisfying \eqref{eq:Alonso_Carneiro_Gamba_coef}. As our
  goal is to control the $L^p$ growth by a high enough moment, we
  will now focus our attention on showing that
  $\norm{f(t)}_{r^\prime(p-1)}$ and $\norm{f(t)}_q$ can both be
  controlled by powers of $M_{p,\gamma}(t)$ and $M_{\bm{r}}(t)$, for
  some explicit $\bm{r}$, in a certain range of parameters. For any
  $1<q<p$ we define
  $$\alpha=-\gamma \frac{q-1}{p-1}<-\gamma,$$
  and find that, for any measurable function $f$:
  \begin{equation*}
    \begin{split}
      \int_{\R^d}\abs{f(v)}^{q}\d v &=
      \int_{\R^d}\frac{\avg{v}^{\alpha}}{\avg{v}^{\alpha}}\abs{f(v)}^{-\frac{\alpha
          p}{\gamma}}\abs{f(v)}^{\frac{\alpha+\gamma}{\gamma}}\d v
      \\
      &\leq \pa{\int_{\R^d}\avg{v}^\gamma \abs{f(v)}^{p}\d
        v}^{-\frac{\alpha}{\gamma}}\pa{\int_{\R^d}\avg{v}^{\frac{\alpha\gamma}{\alpha+\gamma}}\abs{f(v)}\d
        v}^{\frac{\alpha+\gamma}{\gamma}}.
    \end{split}
  \end{equation*}
  Thus, introducing the notation $\mathfrak{a}=-\gamma\frac{q-1}{p-q}$
  \begin{equation}
    \label{eq:L_p_estimates_I}
    \norm{f(t)}_q
    \leq M_{p,\gamma}(t)^{\frac{q-1}{q(p-1)}}
    M_{\mathfrak{a}}(t)^{\frac{p-q}{q(p-1)}}.
\end{equation}
We would like to explore the special case $q=r^\prime(p-1)$, which we
need to verify is possible. We notice that if $r>p$ then
$r^\prime(p-1)<p$. Also, in order for $r^\prime(p-1)$ to be greater
than $1$ we only need that
$$r < \begin{cases} \frac{1}{2-p} & 1<p<2 \\
  \infty & p\geq 2.
\end{cases}$$
Since for $1<p<2$ we always have that $p<1/(2-p)$, a choice of $r>p$
such that $1<r^\prime (p-1)<p$ is indeed always possible. With this
choice, setting $\mathfrak{b}=-\gamma \frac{r(p-2)+1}{r-p}$, we find
that
\begin{equation}
  \label{eq:L_p_estimates_II}
  \norm{f(t)}^{p-1}_{r^\prime(p-1)} \leq M_{p,\gamma}(t)^{1-\frac{r-1}{r(p-1)}}M_{\mathfrak{b}}(t)^{\frac{r-p}{r(p-1)}}.
\end{equation}
Plugging \eqref{eq:L_p_estimates_I}, \eqref{eq:L_p_estimates_II} in
\eqref{eq:L_p_estimates_0} we find that
\begin{equation}
  \label{eq:L_p_estimates_III}
  \frac{\d}{\d t}\frac{1}{p}\norm{f(t)}_p^p
  \leq
  C\,M_{p,\gamma}(t)^{1-\frac{r-q}{rq(p-1)}}
  M_{\mathfrak{a}}(t)^{\frac{p-q}{q(p-1)}}
  M_{\mathfrak{b}}(t)^{\frac{r-p}{r(p-1)}}-c_\gamma
  M_{p,\gamma}(t).
\end{equation}
Since for any $0<\delta<1$ and $a,b>0$, we have that
$\sup_{x>0}\pa{ax^{1-\delta}-bx}
=\pa{\frac{1-\delta}{b}a}^{\frac{1}{\delta}}$,
we see that if $0<\frac{r-q}{rq(p-1)}<1$
\begin{equation}\label{eq:L_p_estimates_IV}
  \begin{gathered}
    \frac{\d}{\d t}\frac{1}{p}\norm{f(t)}_p^p
    \leq
    C_{p,r,q,\gamma} \pa{M_{\mathfrak{a}}(t)^{\frac{p-q}{q(p-1)}}
      M_{\mathfrak{b}}(t)^{\frac{r-p}{r(p-1)}}}^{\frac{rq(p-1)}{r-q}}
    \leq C_{p,r,q,\gamma} M_{\max\pa{ \mathfrak{a},\mathfrak{b} }} (t)^p
  \end{gathered}
\end{equation}
This will give us the desired result as long as we can choose
$1<r,q,\ell<\infty$ such that
$$q<p<r,\quad (2-p)r<1, \quad \frac{1}{q}+\frac{1}{\ell}=1+\frac{\gamma}{d}+\frac{1}{r}, \qquad \text{ and } \qquad \frac{r-q}{rq(p-1)}<1.$$
As $\ell$ only appears in \eqref{eq:Alonso_Carneiro_Gamba} in the norm
$\|\M\|_{\ell}$, we can choose it to be as large as we want. In
particular, for our setting, we can replace
\eqref{eq:Alonso_Carneiro_Gamba_coef} with
\begin{equation}
  \label{eq:L_p_estimates_V} 
  \frac{1}{q}<1+\frac{\gamma}{d}+\frac{1}{r},
\end{equation}
and then choose $\ell$ accordingly. We will choose $r=p+\eta$,
$q=p-\eta$ and see what conditions we must have for $\eta>0$:
\begin{enumerate}[--]
\item For $q>1$ we require that $\eta<p-1$.
\item The condition $(2-p)r<1$ is only valid when $1<p<2$. In that
  case any $\eta<\frac{\pa{p-1}^2}{2-p}$ will do.
\item Demanding that $r-q <rq(p-1)$ in this setting is equivalent to
  $2\eta < \pa{p^2-\eta^2}(p-1)$ which is valid when
$$\eta<\frac{p^2}{\sqrt{\frac{1}{\pa{p-1}^2}+p^2}+\frac{1}{p-1}}.$$
\item Lastly, inequality \eqref{eq:L_p_estimates_V} is equivalent to
  $2\eta < \pa{1+\frac{\gamma}{d}}\pa{p^2-\eta^2}$, which is valid
  when
$$\eta < \frac{p^2}{\sqrt{\frac{1}{\pa{1+\frac{\gamma}{d}}^2}+p^2}+\frac{1}{1+\frac{\gamma}{d}}}.$$
\end{enumerate}
This concludes the proof with the choice
$\bm{r}=\max\pa{\mathfrak{a},\mathfrak{b}}.$
\end{proof}

\subsection{Lower Bounds - Modification of the
  Solution}\label{sec:Carlen_Lu_trick}

In most studies connected to the entropy method for kinetic equations,
a lower bound on the function is needed. In this subsection we will
adapt a method from \cite{Lu} to achieve such a bound by a forced
modification of the solution. We will then investigate the relation
between the entropy and entropy production of it and the original
solution.

Given a non-negative and measurable function $f$ and $0<\delta<1$ we
define
\begin{equation}\label{eq:def_f_delta}
f_\delta(v)=\pa{1-\delta}f(v)+\delta \M(v).
\end{equation}
{A simple, yet important observation is that} if $f$ is integrable and of unit mass, the same occurs for $f_{\delta}.$ Moreover,  $\bm{L}_{\gamma}f_{\delta}=(1-\delta)\bm{L}_{\gamma}f$ since $\bm{L}_{\gamma}(\M)=0.$ 
\begin{lem}\label{lem:entropy_connection}
Let $f\in L^1\pa{\R^{d}}$ be of unit mass. For all $0<\delta<1$ we have that 
\begin{equation}\label{eq:entropy_connection}
H\pa{f|\M} \leq \frac{H\pa{f_\delta|\M}}{1-\delta}+\frac{\delta}{1-\delta}\pa{\log\pa{\frac{1}{\delta}}-\frac{\pa{1-\delta}\log\pa{1-\delta}}{\delta}}
\end{equation}
\end{lem}
\begin{proof}
We start by noticing that since the function $\phi(x)=x\log x$ satisfies
$$\phi(x+y) \geq \phi(x)+\phi(y),$$
for any $x,y>0$, we have that
{\begin{equation}\label{eq:entropy_connection_proof_I}
\begin{gathered}
H\pa{f_\delta}\geq H\pa{\pa{1-\delta}f}+H(\delta \M) = \pa{1-\delta}\log\pa{1-\delta}+\\
\delta \log \delta + \pa{1-\delta}H(f)+\delta H(\M).
\end{gathered}
\end{equation}
where $H(f)=\int_{\R^d}f(v)\log f(v)dv$.} On the other hand as
{\begin{equation}
\label{eq:entropy_connection_proof_II}
\begin{gathered}
H\pa{f|\M}= H(f)-\int_{\R^d} f(v) \log \M(v)\d v\\
=H(f)-H(\M)- \int_{\R^d}\pa{f(v)-\M(v)}\log \M(v)\d v
\end{gathered}
\end{equation}}
we find that 
{\begin{equation}\label{eq:entropy_connection_proof_II.5}
\begin{gathered}
\pa{1-\delta}H(f)+\delta H(\M)=\pa{1-\delta}H\pa{f|\M}+H\pa{\M}\\+\pa{1-\delta}  \int_{\R^d}\pa{f(v)-\M(v)}\log\M(v)\d v\\
=\pa{1-\delta}H\pa{f|\M}+\pa{1-\delta}  \int_{\R^d}f(v)\log \M(v)\d v+\delta H\pa{\M}.
\end{gathered}
\end{equation}}
{Using \eqref{eq:entropy_connection_proof_II} again for $f_\delta$ yields}
\begin{equation}\label{eq:entropy_connection_proof_III}
H\pa{f_\delta}=H\pa{f_\delta|\M}+\pa{1-\delta}\int_{\R^d}f(v)\log\M(v)\d v+\delta H\pa{\M}.
\end{equation}
Thus, combining \eqref{eq:entropy_connection_proof_II.5} and \eqref{eq:entropy_connection_proof_III} yields
$$\pa{1-\delta}H(f)+\delta H(\M)=\pa{1-\delta}H\pa{f|\M}+H\pa{f_\delta}-H\pa{f_\delta|\M}.$$
Plugging this into \eqref{eq:entropy_connection_proof_I} we find that
$$\pa{1-\delta}H\pa{f|\M} \leq H\pa{f_\delta|\M} -\pa{1-\delta}\log\pa{1-\delta} -\delta \log \delta$$ 
from which the result follows.
\end{proof}
The next step {in our study will be} to understand how the entropy dissipation of $f_\delta(t)$ behaves with respect to the entropy dissipation of $f(t)$. From this point onwards we will assume that $\delta=\delta(t)$ is a smooth function of $t$.
\begin{lem}\label{lem:dissipation_connection}
Let $f(t,v)$ be a non-negative solution to \eqref{eq:BE} that has a unit mass and let $g(t,\cdot)=f_{\delta(t)}(t,\cdot)$ be defined as in \eqref{eq:def_f_delta}. Then, if $\delta(t)$ is a non-increasing function, we have that
\begin{equation}\label{eq:dissipation_connection}
\frac{\d}{\d t}H\pa{g(t)} \leq -\D_\gamma \pa{g(t)} -\delta'(t)H\pa{f(t)|\M}+\frac{\delta'(t)\log\pa{\delta(t)}}{1-\delta(t)}
\end{equation}
{where $\delta'(t)=\frac{\d}{\d t}\delta(t).$}
\end{lem}
\begin{proof} To begin with, we notice that $g(t,\cdot)$ solves the following equation:
\begin{equation}\label{eq:equation_for_f_delta}
\partial_t g(t,v)=\bm{L}_{\gamma}g(t,v) - \frac{\delta'(t)}{1-\delta(t)}\pa{g(t,v)-\M(v)}.
\end{equation}
Thus,
\begin{equation}\label{eq:dissipation_connection_proof_I}
\frac{\d}{\d t}H\pa{g(t)|\M} = -\D_\gamma \pa{g(t)} - \frac{\delta'(t)}{1-\delta(t)} \int_{\R^d}\pa{g(t,v)-\M(v)} \log\pa{\frac{g(t,v)}{\M(v)}}\d v\end{equation}
{since} $g(t,v)$ has a unit mass for all $t$. Using the convexity of the relative entropy we see that
$$ \int_{\R^d}g(t,v) \log\pa{\frac{g(t,v)}{\M(v)}}\d v= H\pa{g(t)|\M} \leq \pa{1-\delta(t)}H\pa{f(t)|\M}.$$
Also, since $g(t,v) \geq \delta(t) \M(v)$ 
$$-  \int_{\R^d}\M(v)\log\pa{\frac{g(t,v)}{\M(v)}}\d v \leq \log\pa{\frac{1}{\delta(t)}}.$$
Combining the above with \eqref{eq:dissipation_connection_proof_I} and using the fact that $\delta'(t)\leq 0$ we conclude the result.
\end{proof}
We now have all the tools we need to prove Theorem
\ref{thm:rate_of_convergence}

\section{Algebraic Rate of Convergence to Equilibrium}
\label{sec:rate_of_convergence}

The key to proving Theorem \ref{thm:rate_of_convergence} is the
entropy inequality \eqref{eq:entropy_inequality}. We start the section
with {a couple of simple lemmas that evaluate} the terms in that
inequality.
\begin{lem}
  \label{lem:pre_final_proof}
  Let $f$ be a non-negative function of unit mass.
  \begin{enumerate}[{\bf 1)}]
  \item Let $\mu>0$ and $p>1$. Then for any $\epsilon>0$ there exists
    a uniform constant, $C_{\mu,d,p,\epsilon}>0$, depending only on
    $\mu,d,p$ and $\epsilon$ such that
    {$$\int_{f(t,v)\geq 1}\pa{1+\abs{v}}^{\mu} f(v)\log f(v)\d v \leq
      C_{\mu,d,p,\epsilon}\pa{1+m_{(1+\epsilon)\mu}(f)}^{\frac{1}{1+\epsilon}}\norm{f}_{p}^{\frac{p\epsilon}{1+\epsilon}}.$$}
  \item For any $\mu > 0$, it holds
    $$\int_{\R^d}\pa{1+\abs{v}}^{\mu+2} f(v)\d v\leq 2^{\mu+1}\pa{1+m_{\mu+2}(f)}.$$
  \item If $f(v)\geq A\exp\pa{-B\abs{v}^2}$ for some $A,B>0$ we have
    that
    $$-\int_{\R^d}\bm{K}_\mu f(v)\log f(v)\d v \leq C_{\mu,d,\gamma}\pa{\abs{\log A}(1+m_{\mu}(f))+B(1+m_{\mu+2}(f))}.$$
    where $C_{\mu,d,\gamma}$ is a uniform constant depending only on $\mu,d$ and $\gamma$.
  \end{enumerate}
\end{lem}

\begin{proof}
  To prove $\mathbf{1)}$ we notice that by H\"older's inequality
  {\begin{multline*}
      \int_{f(t,v) \geq 1}\pa{1+\abs{v}}^{\mu} f(v)\log f(v)\d v\\
      \leq \pa{\int_{\R^d}\pa{1+\abs{v}}^{(1+\epsilon)\mu} f(v)\d v}^{\frac{1}{1+\epsilon}}\pa{\int_{f\geq 1}\pa{\log f(v)}^{\frac{1+\epsilon}{\epsilon}}f(v)\d v}^{\frac{\epsilon}{1+\epsilon}}\\
      \leq 2^{\mu}\pa{\sup_{x>1} \abs{\log
          x}^{(1+\epsilon)/\epsilon}x^{1-p}}^{\frac{\epsilon}{1+\epsilon}}\pa{1+m_{(1+\epsilon)\mu}(f)}^{\frac{1}{1+\epsilon}}\norm{f}_{p}^{\frac{p\epsilon}{1+\epsilon}},\end{multline*}}
  showing the desired result. The second point $\mathbf{2)}$ is
  obvious, and to show $\mathbf{3)}$ we notice that under the
  condition on $f$ one has that
$$-\int_{\R^d}\bm{K}_\mu (f)(v)\log f(v)\d v \leq -\log A \int_{\R^d}\bm{K}_\mu f(v)\d v+B\int_{\R^d}\abs{v}^2\bm{K}_\mu f(v)\d v.$$
Now, as 
$$\int_{\R^d}\phi(v)\bm{K}_\mu (v)\d v=\int_{\R^d\times\R^d\times\sph}\abs{v-v_\ast}^\mu b\pa{\cos\theta}f(v)\M(v_\ast)\phi\pa{v^\prime}\d v\d\vb \d\sigma$$
we {find that} 
$$\int_{\R^d}\bm{K}_{\mu} f(v)\d v = \int_{\R^d}\Sigma_\mu(v)f(v)\d v \leq C (1+m_{\mu}(f)), $$
for some uniform constant $C$, due to {Lemma \ref{lem:loss_operator}}. {For} $\phi(v)=\abs{v}^2$ we use the fact that $\abs{v^\prime}^2 \leq \abs{v}^2+\abs{v_\ast}^2$ and {conclude that}
$$\int_{\R^d}\abs{v}^2\bm{K}_\mu (v)\d v \leq C\left(1+ m_{\mu+2}(f)\right).$$
This completes the proof.
\end{proof}

\begin{lem}\label{lem:common_term_in_theorems_of_rate_of_convergence} Assume that $b \in L^{\infty}(\sph)$ with $\|b\|_{1}=1.$
Let $f_0\in L^1\pa{\R^d}$ be a non-negative function with unit mass. Assume in addition that there exists $p>1$ such that $f_0\in L^p\pa{\R^d} \cap L^{1}_{s}(\R^{d})$ for some $s>2\max\pa{1,\abs{\gamma}}$ such that $s\geq \bm{r}$, with $\bm{r}$ as in \eqref{eq:def_of_S}.\\
Consider $\mu,\epsilon>0$ such that $(1+\epsilon)\mu\leq s$ and $\mu+2\leq s$, and let $f(t,v)$ be a non-negative solution to \eqref{eq:BE}. {Define} 
$$g(t,v)=f_{\delta(t)}(t,v)=(1-\delta(t))f(t,v)+\delta(t)\M(v)$$
where $0 \leq \delta(t) \leq 1$ is a smooth decreasing function.  
\begin{enumerate}
\item[\textbf{\textit{(i)}}] Then there exists a uniform constant $C_{0}$, depending only on $d,\gamma,p,s,\mu,\epsilon$ and $\norm{f_0}_p$ and $\norm{f_0}_{L^1_s}$ such that
\begin{equation}\label{eq:common_term_in_theorems_of_rate_of_convergence}
\begin{gathered}
\int_{\R^d}\pa{1+\abs{v}}^\mu g(t,v)\log g(t,v)\d v+\int_{\R^d}\pa{1+\abs{v}}^{\mu+2} g(t,v)\d v \\
\leq C_{0}\pa{1+t}^{\wp}
\end{gathered}
\end{equation}
where $\wp:=\frac{\mu}{s}+\max\pa{\frac{2}{s},\frac{\epsilon}{1+\epsilon}\pa{\frac{\bm{r}p}{s}+1}}.$
\item[\textbf{\textit{(ii)}}] There exists a uniform constant $C_1$, depending only on $d,\gamma,s$ and $\norm{f_0}_{L^1_s}$ such that
\begin{equation}\label{eq:lower_bound_term_in_theorems_of_rate_of_convergence_trick}
\begin{gathered}
-\int_{\R^d}\bm{K}_\mu g(t,v)\log g(t,v) \d v \leq C_{1}\pa{-\log\pa{\delta(t)}+\pa{1+t}^{\frac{\mu+2}{s}}}
\end{gathered}
\end{equation}
\end{enumerate}
\end{lem}
\begin{proof}
Using Corollary \ref{thm:improved_evolution_of_moments}, Theorem \ref{thm:L_p_bounds} and Lemma \ref{lem:pre_final_proof} and the conditions on $\eta$ and $s$ we find that there exists a universal constant $C$, depending on the appropriate parameters and norms, such that
{$$\int_{f(t,v)\geq 1}\pa{1+\abs{v}}^\mu f(t,v)\log f(t,v)\d v+\int_{\R^d}\pa{1+\abs{v}}^{\mu+2} f(t,v)\d v$$}
$$\leq C_1\pa{1+t}^{\frac{\mu}{s}}\pa{1+t}^{\frac{\epsilon}{1+\epsilon}\pa{\frac{\bm{r}p}{s}+1}}
+C_2 \pa{1+t}^{\frac{\mu+2}{s}} $$
showing that \eqref{eq:common_term_in_theorems_of_rate_of_convergence} holds for the solution $f(t,v)$. {Since} $\phi(x)=x\log x$ is convex on $\R_+$
$$g(t)\log g(t) \leq \pa{1-\delta(t)}f(t)\log f(t) + \delta(t) \M \log \M. $$
Thus
\begin{multline*}
\int_{\R^d}\pa{1+\abs{v}}^\mu g(t,v)\log g(t,v)\d v+\int_{\R^d}\pa{1+\abs{v}}^{\mu+2} g(t,v)\d v\\
\leq \int_{f(t,v)\geq 1}\pa{1+\abs{v}}^\mu f(t,v)\log f(t,v) \d v+\int_{\R^d}\pa{1+\abs{v}}^{\mu+2} f(t,v)\d v+C_\M,\end{multline*}
with $C_\M$ independent of $\delta$ or $t$, {concluding the proof of $\textbf{\textit{(i)}}$}\\

To show $\textbf{\textit{(ii)}}$ we remind ourselves that $g(t,v) \geq \delta(t)\M(v)$, and using part $\mathbf{3)}$ of Lemma \ref{lem:pre_final_proof} we find that
\begin{multline*}
-\int_{\R^d}\bm{K}_\mu g(t,v)\log g(t,v)\d v \leq  C_{\mu,d,\gamma}\pa{-\log\pa{\delta(t)}+m_{\mu+2}\pa{g(t)}}\\
=C_{\mu,d,\gamma}\pa{-\log\pa{\delta(t)}+\pa{1-\delta(t)}m_{\mu+2}\pa{f(t)}+\delta(t)}
\leq C_{\mu,d,\gamma}\pa{-\log\pa{\delta(t)}+\pa{1+t}^{\frac{2+\mu}{s}}}.\end{multline*}
\end{proof}
To complete the proof of Theorem \ref{thm:rate_of_convergence}, we
need the following Lemma which is reminiscent of \cite[Lemma
7.2]{Lu}. As the proof is an easy adaptation, we omit the details
here.
\begin{lem}
  \label{lemapp:diff_ineqaulity}
  Let $\alpha,\beta>0$ such that $\alpha<1$. Consider the differential
  inequality
  $$u^\prime(t) \leq -C\pa{1+t}^{-\alpha}u(t)^{1+\beta}+\xi(t), \qquad t \geq 0.$$
  If $u(t)$ is an absolutely continuous function satisfying the above,
  and if
  $$C_{\xi}=\sup_{t\geq 0}\pa{1+t}^{\frac{\beta+1-\alpha}{\beta}}\xi(t)<\infty$$
  then
  $$u(t)\leq
  \max\pa{1,u(0),\pa{\frac{1-\alpha+\beta C_\xi}{\beta C}}^{\frac{1}{\beta}}}\pa{1+t}^{-\frac{1-\alpha}{\beta}}.$$
\end{lem}

We are finally ready to prove our main theorem. 

\begin{proof}[Proof of Theorem \ref{thm:rate_of_convergence}]
  As the previous lemmas indicate, we start by identifying
  $s_{p,d,\gamma}$ in the theorem as $\bm{r}$, defined in
  \eqref{eq:def_of_S}. We start by choosing $\epsilon,\mu>0$ small
  enough such that $(1+\epsilon)\mu \leq s-2$ and
  $$\frac{\epsilon}{1+\epsilon}\pa{\frac{\bm{r}p}{s}+1}\leq \frac{2}{s}.$$
  In such a case, $\wp=\frac{\mu+2}{2}$ where $\wp$ is defined in Lemma \ref{lem:common_term_in_theorems_of_rate_of_convergence}. Consider the function
  $$\delta(t)=\frac{1}{2} \exp\pa{-\pa{1+t}^{\frac{\mu+2}{s}}}$$
  As before, set $g(t,v)=f_{\delta}(t,v)=(1-\delta(t))f(t,v)+\delta(t)\M(v).$ 
  Using Theorem \ref{thm:entropy_inequality} with Lemma \ref{lem:common_term_in_theorems_of_rate_of_convergence}, we find that there exists a uniform constant, $C=C(f_{0},d,\gamma,p,s,\mu)$ that depends on the appropriate parameters and norms, as well as $\lambda_0$ from Theorem \ref{thm:entropy_method_hard_potentials}, such that
  \begin{equation*}\begin{split}
      \D_\gamma \pa{g(t)} &\geq C\pa{\pa{1+t}^{\frac{\mu+2}{s}}+\pa{1+t}^{\frac{2}{s}}}^{\frac{\gamma}{\mu}}H\pa{g(t)|\M}^{1-\frac{\gamma}{\mu}}\\
      &\geq 2^{\frac{\gamma}{\mu}}C \pa{1+t}^{\frac{\mu+2}{s}\frac{\gamma}{\mu}}H\pa{g(t)|\M}^{1-\frac{\gamma}{\mu}}.\end{split}\end{equation*}
  Combining the above with Lemma \ref{lem:dissipation_connection}, and using the fact that $H(f(t)|\M) \leq H(f_{0}|\M)$ for any $t \geq 0$, we find that
  $$
  \frac{\d}{\d t}H\pa{g(t)|\M} \leq -C_{0,d,\gamma,p,s,\mu}\pa{1+t}^{\frac{\mu+2}{s}\frac{\gamma}{\mu}}H\pa{g(t)|\M}^{1-\frac{\gamma}{\mu}}+\xi(t)$$
  where we introduced
  $$\xi(t)=\frac{\mu+2}{2s}(1+t)^{\frac{\mu+2}{s}-1}\exp\pa{-\pa{1+t}^{\frac{\mu+2}{s}}}\pa{H\pa{f_0|\M}+2\pa{1+t}^{\frac{(\mu+2)}{s}}}.$$
  From the above differential inequality, applying Lemma \ref{lemapp:diff_ineqaulity} with 
  $$\alpha=-\frac{\mu+2}{s}\frac{\gamma}{\mu}, \qquad \beta=-\frac{\gamma}{\mu}$$
  we see that, provided $\alpha <1$, there exists a constant $\mathbf{C} >0$ that depends only on the appropriate parameters such that
  $$H\pa{g(t)|\M} \leq \mathbf{C}\pa{1+t}^{-\frac{1-\alpha}{\beta}} \qquad \forall t \geq 0.$$
  {Choosing
    $$\mu=\frac{s-2}{1+\epsilon},$$
    for an appropriate $\epsilon$ sufficiently small, in order to maximise the convergence rate, we see that $\alpha=\tfrac{|\gamma|}{s-2}+ 2\epsilon\tfrac{|\gamma|}{s(s-2)}<1$ provided that $|\gamma| < s-2$ and $\epsilon$ small enough. This is indeed valid in our setting and} In that case, {$-\frac{1-\alpha}{\beta}=\frac{s-2}{(1+\epsilon)\gamma}+\frac{s+2\epsilon}{(1+\epsilon)s}$ and
    $$H\pa{g(t)|\M} \leq \mathbf{C}\pa{1+t}^{\frac{s-2}{(1+\epsilon)\gamma}
      +\frac{s+2\epsilon}{(1+\epsilon)s}} \qquad \forall t \geq 0$$}
  for some constant $\mathbf{C} > 0$ depending only on $\|f_{0}\|_{p}$, $\|f_{0}\|_{L^{1}_{s}}$ and the parameters. Using the above with \eqref{eq:entropy_connection} we conclude that
\begin{multline*}
    H\pa{f(t)|\M} \leq 2\mathbf{C}\pa{1+t}^{\frac{s-2}{(1+\epsilon)\gamma}
      +\frac{s+2\epsilon}{(1+\epsilon)s}}\\
    +\exp\pa{-\pa{1+t}^{\frac{s+2\epsilon}{(1+\epsilon)s}}}\pa{\pa{1+t}^{\frac{s+2\epsilon}{(1+\epsilon)s}}+\sup_{0<x<1}\frac{(x-1)\log\pa{1-x}}{x}}.
  \end{multline*}
concluding the proof.
\end{proof}

Using the decay rate in Theorem \ref{thm:rate_of_convergence} one can
easily obtain by interpolation the boundedness of moments and $L^p$
norms, using a technique sometimes known as ``slowly growing \textit{a priori}
bounds'' (see for example \cite{DesvMou,ToVi}). We start by showing that
moments are uniformly bounded in time, if a sufficiently high moment
is initially bounded:

\begin{thm}
  \label{thm:moments-bounded}
  Take $p > 1$ and $-d < \gamma <0$, and let
  $f_0\in L^1_s\pa{\R^d} \cap L^p(\R^d)$ be a non-negative function
  with unit mass, for some $s \geq s_{p,d,\gamma}$ (where
  $s_{p,d,\gamma}$ is the constant from Theorem
  \ref{thm:rate_of_convergence}). Let $f = f(t)$ be the solution to
  equation \eqref{eq:BE} with a bounded angular kernel $b$.

Given $k > 0$, there exists $\beta > k$ depending only on
  $k, d, s, \gamma$ such that if additionally we have
  $M_\beta(f_0) < +\infty$ then it holds that
  \begin{equation*}
    M_k(f(t))
    \leq
    C_k
    \qquad \text{for all $t \geq 0$},
  \end{equation*}
  for some constant $C_k > 0$ that depends only on $k$, $p$, $s$, $d$,
  $\gamma$, $M_\beta(f_0)$ and $\|f_0\|_p$.
\end{thm}

\begin{proof}
  According to Theorem \ref{thm:evolution_of_moments} we have
  \begin{equation*}
    M_\beta(f(t)) \leq C_\beta (1+t)
    \qquad \text{for $t \geq 0$.}
  \end{equation*}
  On the other hand, fixing $0 < \sigma < -1 + \frac{s-2}{|\gamma|}$
  (for definiteness, take
  $\sigma := -\frac{1}{2} + \frac{s-2}{2|\gamma|}$), we can apply
  Theorem \ref{thm:rate_of_convergence} and the Csiszár-Kullback
  inequality to get
  \begin{equation}
    \label{eq:mbproof1}
    \| f(t) - \M\|_1^2 \leq H(f(t) | \M)
    \leq C_0 (1+t)^{-\sigma}
    \qquad \text{for $t \geq 0$.}
  \end{equation}
  Now, by interpolation, for $\theta \in (0,1)$ given by $\beta(1-\theta) = k,$ that is
  \begin{equation*}
    \theta = \frac{\beta - k}{\beta},
  \end{equation*}
  we have
  \begin{multline*}
    M_k(|f(t) - \M|)
    \leq
    \|f(t) - \M\|_1^\theta
    \, M_\beta(|f(t) - \M|)^{1-\theta},
    \\
    \leq
    C_0^{\frac{\theta}{2}}
    (1+t)^{-\frac{\theta \sigma}{2}}
    \, \big( M_\beta(f(t)) + M_\beta(\M) \big)^{1-\theta}
    \\
    \leq
    C
    (1+t)^{-\frac{\theta \sigma}{2}}
    (1+t)^{1-\theta}
    = C (1+t)^{-\frac{\theta \sigma}{2} + 1 - \theta},
  \end{multline*}
  for some $C > 0$ depending on the allowed quantities. Taking $\beta$
  large, $\theta$ becomes close to $1$ and we can choose $\beta$ so
  that
  \begin{equation*}
    -\frac{\theta \sigma}{2} + 1 - \theta < 0,
  \end{equation*}
  which corresponds to $\beta$ satisfying
   $ \beta > \frac{k(2 + \sigma)}{\sigma}.$
  This gives
  \begin{equation*}
    M_k(|f(t) - \M|)
    \leq C
    \qquad \text{for all $t \geq 0$,}
  \end{equation*}
  which gives the result since $ M_k(f(t)) \leq M_k(|f(t)-\M|) + M_k(\M) \leq C + M_k(\M).$
\end{proof}

We turn now to the boundedness of $L^p$ norms:

\begin{thm}
  \label{thm:Lp-bounded}
  Take $p > 1$ and $-d < \gamma < 0$, and let
  $f_0\in L^1_s\pa{\R^d} \cap L^p(\R^d)$ be a non-negative function
  with unit mass, for some $s \geq s_{p,d,\gamma}$ (where
  $s_{p,d,\gamma}$ is the constant from Theorem
  \ref{thm:rate_of_convergence}). Let $f = f(t)$ be the solution to
  equation \eqref{eq:BE} with a bounded angular kernel $b$.

Given $q > 0$, there exists $r > 1$ depending only on
  $q, d, s, \gamma$ such that if additionally we have
  $\|f_0\|_r < +\infty$ then it holds that
  \begin{equation*}
    \|f(t)\|_q
    \leq
    C_q
    \qquad \text{for all $t \geq 0$},
  \end{equation*}
  for some constant $C_q > 0$ that depends only on $q$, $p$, $s$, $d$,
  $\gamma$ and $\|f_0\|_r$.
\end{thm}

\begin{proof}
  The proof is similar to the previous one. Using Theorems
  \ref{thm:L_p_bounds} and \ref{thm:evolution_of_moments} we have
  \begin{equation*}
    \|f(t)\|_r \leq C (1+t)^2
    \qquad \text{for $t \geq 0$,}
  \end{equation*}
  for some $C > 0$ depending on the allowed quantities. By
  interpolation, for $\theta \in (0,1)$ given by
  \begin{equation*}
    \theta = \frac{q-r}{q(r-1)},
  \end{equation*}
  we have, using also \eqref{eq:mbproof1} (with the same choice of
  $\sigma$),
  \begin{multline*}
    \|f(t) - \M\|_q
    \leq
    \|f(t) - \M\|_1^\theta
    \, \|f(t) - \M\|_r^{1-\theta},
    \\
    \leq
    C_0^{\frac{\theta}{2}}
    (1+t)^{-\frac{\theta \sigma}{2}}
    \, \big( \|f(t)\|_r + \|\M\|_r \big)^{1-\theta}
    \\
    \leq
    C
    (1+t)^{-\frac{\theta \sigma}{2}}
    (1+t)^{2(1-\theta)}
    = C (1+t)^{-\frac{\theta \sigma}{2} + 2 (1-\theta)},
  \end{multline*}
  for some $C > 0$ depending on the allowed quantities only. Taking
  $r$ large, $\theta$ approaches $1$ so we can choose $r$ so that
  \begin{equation*}
    -\frac{\theta \sigma}{2} + 2(1-\theta) < 0.
  \end{equation*}
  With this choice,
  \begin{equation*}
    \|f(t) - \M\|_q
    \leq C
    \qquad \text{for all $t \geq 0$,}
  \end{equation*}
  which proves the result by noticing that
  $\|f(t)\|_q \leq \|f(t)-\M\|_q + \|\M\|_q$.
\end{proof}

\begin{rem}
  The previous bounds can be now used in the proof of Theorem
  \ref{thm:rate_of_convergence} to improve the decay exponent. We do
  not give the details of this improvement since we do not believe it
  to be optimal, and the exponent $\sigma$ depends anyway on
  $s_{p,d,\gamma}$, which has a complicated explicit expression.
\end{rem}

\section{Instantaneous Generation of Maxwellian Lower
  Bounds}
\label{sec:maxwellian_lower_bounds}

In this Section we will investigate the phenomena of instantaneous
creation of a Maxwellian lower bound to the solution of our linear
Boltzmann equation, a property that is well understood for the fully
non-linear Boltzmann equation. We arrive at this result by a careful
investigation of the gain operator, $\bm{K}_{\gamma}$. The following
Lemma, whose proof is left to Appendix \ref{app:carleman}, is the
first step in this direction.

\begin{lem}
  \label{lem:little_k_repres}
  For a collision kernel of the form \eqref{eq:Bcs}, the gain part
  operator $\bm{K}_{\gamma}=\Q^{+}(\cdot,\M)$ admits the following
  representation:
  $$\bm{K}_{\gamma}f(v)=\int_{\R^{d}}k_{\gamma}(v,w)f(w)\d w,
  \qquad v,w \in \R^d,$$
  with
  \begin{equation}
    \label{eq:little_k_repres}
    k_\gamma(v,w) = \frac{2^{d-1}}{\pa{2\pi}^{\frac{d}{2}}\abs{v-w}}
    \exp\left(
      -\frac{1}{8}\pa{\abs{v-w}+\frac{\abs{v}^2-\abs{w}^2}{\abs{v-w}}}^2
    \right)
    \int_{\pa{v-w}^{\perp}}\xi_{b,\gamma}\pa{z,v,w} \d z
  \end{equation}
  where 
  \begin{equation}
    \label{eq:def_of_xi}
    \xi_{b,\gamma}(z,v,w) =
    \exp\left(-\frac{\abs{V_\perp+z}^{2}}{2}\right)
    b\pa{\frac{\abs{z}^2-\abs{v-w}^2}{\abs{z}^2+\abs{v-w}^2}}
    \abs{z-(v-w)}^{\gamma-(d-2)},
  \end{equation}
  with $V_\perp$ being the projection of $V=\frac{v+w}{2}$ on the
  subspace that is perpendicular to $v-w$.
\end{lem}

In what follows, we will assume that $\gamma\in(-d,0]$ is given. The
key ingredient in establishing the creation of a lower bound is in
estimating the term {\begin{equation}\label{eq:def_of_I}
    \mathcal{I}_{b}(v,w)=\int_{\pa{v-w}^{\perp}}\xi_{b,\gamma}\pa{z,v,w}\d
    z,
\end{equation}}
which will be the purpose of our next lemma. For $b=1$, we simply use the notation $\mathcal{I}(v,w)$ to denote $\mathcal{I}_{b}(v,w)$.
\begin{lem}\label{lem:estimates_on_I}
Consider $\mathcal{I}_{b}(v,w)$ as defined in \eqref{eq:def_of_I}. Then
\begin{enumerate}[(i)]
\item If $b=1$ {and} $\beta\leq 0$ {then}
\begin{equation}\label{eq:estimation_of_I_when_b=1}
\abs{v-w}^{\beta}\mathcal{I}(v,w) \geq C_{d,\gamma,\beta}\exp\left(-\left(\abs{v}^2+\abs{w}^2\right)\right).
\end{equation}
where $C_{d,\gamma,\beta}>0$ is a universal constant depending only on $d,\gamma$ and $\beta$. 
\item If $b(x) \geq b_0\pa{1-\abs{x}^2}^{\frac{\nu}{2}}$ for some $b_0>0$, {$0\leq \nu \leq 1$, and if $d\geq 2$} then 
\begin{equation}\label{eq:estimation_of_I_when_b_special}
\abs{v-w}^{-1}\mathcal{I}_{b}(v,w) \geq C_{d,\gamma,b_0,\nu}\exp\left(-\frac{\pa{2\nu+d-\gamma-2}\pa{\abs{v}^2+\abs{w}^2}}{\pa{d-\gamma-2}}\right),
\end{equation}
where $C_{d,\gamma,b_0,\nu}>0$ is a universal constant depending only on $d,\gamma,b_0$ and $\nu$. 
\end{enumerate} 
\end{lem}

\begin{proof}
As $z\perp v-w$ we have that
$$\abs{v-w}^\beta \geq \pa{\abs{z}^2+\abs{v-w}^2}^{\frac{\beta}{2}}=|z-(v-w)|^{\beta}.$$
Since $\abs{V_\perp} \leq \abs{V}\leq \frac{\abs{v}+\abs{w}}{2}$, and since $\abs{x\pm y}^2 \leq 2\pa{\abs{x}^2+\abs{y}^2}$ we find that
$$\exp\left(-\frac{\abs{V_\perp+z}^2}{2}\right) \geq \exp\left(-\frac{\abs{v}^2+\abs{w}^2}{2}\right)\exp\left(-\abs{z}^2\right).$$
As such
$$\abs{v-w}^{\beta}\mathcal{I}(v,w)  \geq \exp\left(-\frac{\abs{v}^2+\abs{w}^2}{2}\right) \int_{(v-w)^\perp} \exp\left(-\abs{z}^2\right)\pa{\abs{z}^2+\abs{v-w}^2}^{\frac{\gamma-d+2+\beta}{2}} \d z.$$
If $\beta\geq d-2-\gamma$ then 
$$\abs{v-w}^{\beta}\mathcal{I}(v,w)  \geq \exp\left(-\frac{\abs{v}^2+\abs{w}^2}{2}\right) \int_{\R^{d-1}} \abs{z}^{\gamma-d+2+\beta} \exp\left(-\abs{z}^2\right)\d z,$$
while if $\beta< d-2-\gamma$ then, for a given $\epsilon>0$, we can find a universal constant $C_{\epsilon,d,\gamma,\beta}$ such that 
$$\pa{\abs{z}^2+\abs{v-w}^2}^{\frac{\gamma-d+2+\beta}{2}} \geq C_{\epsilon,d,\gamma,\beta} \exp\left(-\epsilon\pa{\abs{z}^2+\abs{v-w}^2}\right),$$
from which we find that
$$\abs{v-w}^{\beta}\mathcal{I}(v,w)  \geq C_{\epsilon,d,\gamma,\beta}\exp\left(-\frac{\abs{v}^2+\abs{w}^2}{2}\right) \exp\left(-\epsilon\abs{v-w}^2\right)\int_{\R^{d-1}}\exp\left(-(1+\epsilon)\abs{z}^2\right)\d z,$$
completing the proof of $(i)$ with the choice of $\epsilon=\frac{1}{4}$.\\
To show $(ii)$ we start by noticing that
$$b\pa{\frac{\abs{z}^2-\abs{v-w}^2}{\abs{z}^2+\abs{v-w}^2}} \geq 2^{\nu}b_0\frac{\abs{z}^\nu\abs{v-w}^{\nu} }{\pa{\abs{z}^2+\abs{v-w}^2}^{\nu}}.$$
Next, using H\"older inequality we find that for any $0<\alpha<1$, to be chosen at a later stage, one has that
\begin{equation}\label{eq:special_b_interpolation_I}
\begin{split}
\mathcal{I}(v,w) &= \int_{(v-w)^\perp} \xi_{\gamma}(z,v,w)b\pa{\frac{\abs{z}^2-\abs{v-w}^2}{\abs{z}^2+\abs{v-w}^2}}^{\alpha}b\pa{\frac{\abs{z}^2-\abs{v-w}^2}{\abs{z}^2+\abs{v-w}^2}}^{-\alpha}\d z\\
&\leq \mathcal{I}_{b}(v,w)^{\alpha}\mathcal{I}_{b^{-\alpha/(1-\alpha)}}(v,w)^{1-\alpha}.
\end{split}
\end{equation}
Due to the lower bound on $b$ we have that
\begin{multline*}
\mathcal{I}_{b^{-\alpha/(1-\alpha)}}(v,w) \leq \pa{2^\nu b_0}^{-\frac{\alpha}{1-\alpha}}\abs{v-w}^{-\frac{\alpha \nu}{1-\alpha}}\\
\int_{(v-w)^\perp} \exp\left(-\frac{\abs{V_\perp+z}^2}{2}\right)\pa{\abs{z}^2+\abs{v-w}^2}^{\frac{\gamma-d+2}{2}+\frac{\alpha\nu}{1-\alpha}}\abs{z}^{-\frac{\alpha \nu}{1-\alpha}}\d z.\end{multline*}
We now choose
\begin{equation}\label{eq:alpha}
\alpha=\frac{d-2-\gamma}{2\nu+d-2-\gamma},\end{equation}
which satisfies that $0<\alpha<1$ as well as $\frac{\alpha \nu}{1-\alpha}=\frac{d-2-\gamma}{2}.$
With this in hand we get that
$$\mathcal{I}_{b^{-\alpha/(1-\alpha)}}(v,w) \leq \pa{2^\nu b_0}^{-\frac{\alpha}{1-\alpha}}\abs{v-w}^{-\frac{\alpha \nu}{1-\alpha}} \int_{(v-w)^\perp}\exp\left(-\frac{\abs{V_\perp+z}^2}{2}\right) \abs{z}^{-\frac{d-2-\gamma}{2}}\d z.$$
Splitting the integral according to $|z|>1$ or $|z| \leq 1$, it is easy to see that
$$\sup_{v,w}\int_{(v-w)^\perp}\exp\left(-\frac{\abs{V_\perp+z}^2}{2}\right) \abs{z}^{-\frac{d-2-\gamma}{2}}\d z \leq C_{d,\gamma}$$
for some positive constant $C_{d,\gamma} >0$ depending only on $d$ and $\gamma >-d$. Then, there is some positive constant $C$ (depending on $d,\gamma,\nu$ and $b_{0}$) such that 
$$\mathcal{I}_{b^{-\alpha/(1-\alpha)}}(v,w) \leq C\,\abs{v-w}^{-\frac{\alpha \nu}{1-\alpha}}.$$ 
 Going back to \eqref{eq:special_b_interpolation_I}, we find that
$$\abs{v-w}^{-1}\mathcal{I}_{b}(v,w) \geq C_{d,\gamma,\nu,b_{0}}\abs{v-w}^{\nu-1}\mathcal{I}(v,w)^{\frac{1}{\alpha}}$$
for some $C_{d,\gamma,\nu,b_{0}} > 0$. The result now follows from $(i)$ if $\nu\leq 1$ where we recall that $\alpha$ is given by \eqref{eq:alpha}.
\end{proof}

\begin{rem}
  It is interesting to notice that the above constant
  $C_{d,\gamma,\nu,b_{0}}$ can be written as
  $C_{d,\gamma,\nu,b_{0}}=C_{d,\gamma,\nu}\,b_{0}$ for some universal
  constant $C_{d,\gamma,\nu}$ depending only on $d,\gamma$ and
  $\nu \in [0,1].$
\end{rem}

\begin{cor}\label{cor:lower_bound_on_k}
Assume {that $d\geq 2$ and} that the angular kernel $b(\cdot)$ satisfies
\begin{equation}\nonumber
b(x) \geq b_0\pa{1-\abs{x}^2}^{\frac{\nu}{2}}, \qquad x \in (-1,1)
\end{equation} 
for some $b_0>0$ and $0\leq \nu \leq 1$. Then, for all $v,w \in\R^{d}$, 
$$k_\gamma(v,w) \geq C_{d,\gamma,b_0,\nu}\exp\pa{-\mathfrak{\lambda_{1}} \abs{v}^2}\exp\pa{-\lambda_{2}\abs{w}^2},$$
where $C_{d,\gamma,b_0,\nu}>0$ is a universal constant depending only on $d,\gamma,b_0$ and $\nu$ and 
$$\lambda_{1}=\frac{3}{4}+\frac{2\nu+d-\gamma-2}{d-\gamma-2} > 0, \qquad \lambda_{2}=\frac{1}{4}+\frac{2\nu+d-\gamma-2}{d-\gamma-2}>0.$$
\end{cor}
\begin{proof}
We start by noticing that $\abs{v-w}^2\leq 2\pa{\abs{v}^2+\abs{w}^2},$
and
$$\frac{\pa{\abs{v}^2-\abs{w}^2}^2}{\abs{v-w}^2} = \pa{\frac{\abs{(v-w)(v+w)}}{\abs{v-w}}}^2 \leq \abs{v+w}^2 \leq 2\pa{\abs{v}^2+\abs{w}^2}.$$
As such
\begin{equation*}\begin{split}
\frac{1}{4}\pa{\abs{v-w}+\frac{\abs{v}^2-\abs{w}^2}{\abs{v-w}}}^2&=\frac{\pa{\abs{v}^2-\abs{w}^2}^2}{4\abs{v-w}^2}+\frac{\abs{v}^2-\abs{w}^2}{2}+\frac{\abs{v-w}^2}{4}\\
&\leq \abs{v}^2+\abs{w}^2 + \frac{\abs{v}^2-\abs{w}^2}{2}= \frac{3\abs{v}^2}{2}+ \frac{\abs{w}^2}{2}.\end{split}\end{equation*}
The result now follows from \eqref{eq:little_k_repres} and Lemma \ref{lem:estimates_on_I}.
\end{proof}
We are now ready to prove the main theorem of this section:
\begin{thm}\label{thm:instantaneous_creation_of_maxwellian_lower_bound}
Let $f_0\in L^1\pa{\R^d}$ be a non-negative function with unit mass and finite second moment. Let $f(t,v)$ be a non-negative solution to \eqref{eq:BE} with angular kernel $b$ that satisfies
\begin{equation}\nonumber
b(x) \geq b_0\pa{1-\abs{x}^2}^{\frac{\nu}{2}}, \qquad x \in (-1,1)
\end{equation} 
for some $b_0>0$ and $0\leq \nu \leq 1$. Then, {if $d\geq 2$,} there exists a constant $C_{d,\gamma,b_0,\nu}>0$, depending only on $d,\gamma,b_0$ and $\nu$, such that for any $s>0$, $v\in\R^d$ and $t>t_0>0$ we have that
\begin{equation}\label{eq:instantaneous_creation_of_maxwellian_lower_bound}
f(t,v) \geq C_{d,\gamma,b_0,\nu}\pa{1-\exp\pa{-\norm{\Sigma_\gamma}_{\infty} t_0}}\exp\left({-\lambda_{1}\pa{\abs{v}^2+\sup_{\tau \leq t}\pa{2m_s(\tau)}^{\frac{2}{s}}}}\right),
\end{equation} 
where $C_{d,\gamma,b_0,\nu}$ is a constant that depends only on $d,\gamma,b_0$ and $\nu$, and $\lambda_{1}$ is defined in Corollary \ref{cor:lower_bound_on_k}.
\end{thm}
\begin{proof}
As $f(t,v)$ is the solution to \eqref{eq:BE} we find that
$$\partial_t f(t,v)+\Sigma_\gamma(v)f(t,v) = \bm{K}_\gamma(f)(v).$$
Using Lemma {\ref{lem:little_k_repres}} and Corollary \ref{cor:lower_bound_on_k} we can conclude that
\begin{equation*}\begin{split}
\partial_t f(t,v) + \|\Sigma_{\gamma}\|_{\infty} f(t,v) &\geq C_{d,\gamma,b_0,\nu}\exp\left(-\lambda_{1}\abs{v}^2\right)\int_{\R^d}\exp\left(-\lambda_{2}\abs{w}^2\right)f(t,w)\d w\\
&\geq C_{d,\gamma,b_0,\nu}\exp\left(-\lambda_{1}\pa{\abs{v}^2+R^2}\right)\int_{\abs{w}<R}f(t,w)\d w,\end{split}\end{equation*}
for any $R > 0$. For any $s>0$, we know that
$$\int_{\abs{w}<R}f(t,w)\d w =1 - \int_{\abs{w}\geq R}f(t,w)\d w \geq 1- \frac{1}{R^{s}}\int_{\abs{w}>R}\abs{w}^sf(t,w)\d w \geq 1-\frac{m_s(t)}{R^s}.$$
Using the above, and choosing $R=\pa{2m_s(t)}^{\frac{1}{s}}$, we find that for any $s>0$
\begin{equation}\label{eq:instantaneous_creation_I}
\partial_t f(t,v) + \|\Sigma_{\gamma}\|_{\infty} f(t,v) \geq \frac{C_{d,\gamma,b_0,\nu}}{2}\exp\left(-\lambda_{1}\pa{\abs{v}^2+\pa{2m_s(t)}^{\frac{2}{s}}}\right).
\end{equation}
Solving the above inequality and using that $f_{0}$ is nonnegative yields the result. 
\end{proof}

\begin{rem}
  \label{rem:intial_conditions_with_maxwellian_lower_bound}
  Note that if there exists $a,B>0$ such that
  $f_0(v) \geq B\exp\pa{-a\abs{v}^2}$ then solving the differential
  inequality \eqref{eq:instantaneous_creation_I} yields now
$$f(t,v) \geq C\,\exp\pa{-\lambda\pa{\abs{v}^2+\sup_{\xi \leq t}\pa{2m_s(\xi)}^{\frac{2}{s}}}}$$
for some explicit $C=C(a,B,d,\gamma,b_0,\nu)$ and $\lambda=\lambda(d,\gamma,\nu,a)$ and all $t\geq 0$.
\end{rem}

A simple consequence of the above estimate is the following 
\begin{cor}\label{lem:lower_bound_term_in_theorems_of_rate_of_convergence-1}
Under the assumption of Theorem \ref{thm:instantaneous_creation_of_maxwellian_lower_bound}, the nonnegative solution $f(t,v)$ to \eqref{eq:BE} with a bounded angular kernel $b$ is such that, for any $t_0>0$ {and $\mu>0$,
\begin{equation}\label{eq:lower_bound_term_in_theorems_of_rate_of_convergence_no_trick}
\begin{gathered}
-\int_{\R^d}\bm{K}_\mu f(t,v)\log f(t,v)	\d	v \leq C_{2}\pa{1+t}^{\frac{2+\mu}{s}},
\end{gathered}
\end{equation}}
where $C_{2}$ is a uniform constant depending only on $d,\gamma,b_0,\nu,s$ and $t_0$. If there exists $A,B>0$ such that $f_0(v)\geq A\exp\pa{-B\abs{v}^2}$ then the above is valid from $t_0=0$ and the constant will also depend on $A$ and $B$. 
\end{cor}
\begin{proof}
{The proof follows immediately from Theorem \ref{thm:evolution_of_moments}, Lemma \ref{lem:pre_final_proof} and Theorem \ref{thm:instantaneous_creation_of_maxwellian_lower_bound}.}
\end{proof}

\section{Stretched-Exponential Rate of Convergence to equilibrium}
\label{sec:exp}

At this section we will investigate the rate of decay for equilibrium
under the additional assumption of having an exponential moment. We
start by noticing the following simple result which we deduce from
Lemma \ref{lem:lyapunov_functionals} for the convex function
$\mathbf{\Phi}(x)=x^{p}$.
\begin{propo}\label{cor:H_p_decreases}
Let $p>1$ and consider the functional
$$H_p(f)=\int_{\R^d}\M(v)^{1-p}\abs{f(v)}^p\d v.$$ Then, if $H_p(f_0)<\infty$ we have that any non-negative solution $f(t,v)$ to \eqref{eq:BE} with initial data $f_0$ satisfies
$$H_p(f(t)) \leq H_p(f_0) <\infty \qquad \forall t \geq 0.$$
\end{propo}
We will now want to explore how the above $H_p$ can improves our rate of convergence to equilibrium. We start by improving the interpolation inequality between $\D_\gamma$ and $\D_0$ provided by inequality \eqref{eq:entropy_inequality_proof_I}:
\begin{lem}\label{lem:dissipation_interpolation_exp}
For a given $a>0$ and $q\geq 1$ define
$$\mathbf{\Gamma}_{a,q}(f)=\frac{1}{2}\int_{\R^d\times\R^d\times\sph}b\pa{\cos\theta}\exp\pa{a\abs{v-v_\ast}^q}\M(v)\M(v_\ast)\Psi\pa{h(v),h\pa{v^\prime}}
\d v\d\vb\d\sigma,$$
with $\Psi(x,y)=\pa{x-y}\log\pa{x/y}$ {and $h=\frac{f}{\M}$}. Then for any $\gamma<0$ one has that
\begin{equation}\label{eq:dissipation_interpolation_exp}
\D_\gamma(f) \geq  \frac{a^{\frac{\abs{\gamma}}{q}}}{2}\D_0(f)\pa{\log\dfrac{2\mathbf{\Gamma}_{a,q}(f)}{\D_0(f)}}^{\frac{{\gamma}}{q}}
\end{equation}
\end{lem}
\begin{proof}
For a given $R>0$ we set $\mathcal{Z}_{a,R}=\left\{(v,\vb) \in \R^{d}\times \R^{d}\;;\;\abs{v-\vb} \leq \pa{\frac{R}{a}}^{\frac{1}{q}}\right\}$ and denote by $\mathcal{Z}_{a,R}^{c}$ its complementary in $\mathbb{R}^{2d}.$ We have that
\begin{equation*}\begin{split}
\D_0(f) &= \frac{1}{2}\int_{\mathcal{Z}_{a,R}\times \sph}b\pa{\cos\theta}\abs{v-v_\ast}^{\abs{\gamma}}\abs{v-v_\ast}^\gamma \M(v)\M(v_\ast)\Psi\pa{h(v),h\pa{v^\prime}}\d v\d \vb \d\sigma\\
&\phantom{++}+\frac{1}{2}\int_{\mathcal{Z}_{a,R}^{c}\times \sph}b\pa{\cos\theta}\exp\pa{-a\abs{v-v_\ast}^q}\exp\pa{a\abs{v-v_\ast}^q} \\
&\phantom{+++++++} \times
\M(v)\M(v_\ast)\Psi\pa{h(v),h\pa{v^\prime}}\d v\d \vb \d\sigma\\
&\leq \pa{\frac{R}{a}}^{\frac{\abs{\gamma}}{q}}\D_\gamma (f) + \exp(-R)\mathbf{\Gamma}_{a,q}(f).\end{split}\end{equation*}
{We also notice that} for any $a,q>0$ we have that $1\leq \exp\pa{a\abs{v-v_\ast}^q}$ {and as such} $\D_0(f) \leq \mathbf{\Gamma}_{a,q}(f)$. {Thus, the choice
$$R= \log\pa{\frac{2\mathbf{\Gamma}_{a,q}(f)}{\D_0(f)}}>\log 2 >0,$$
is valid and yields}
$$\frac{\D_0(f)}{2} \leq a^{\frac{\gamma}{q}}\pa{\log\pa{\dfrac{2\mathbf{\Gamma}_{a,q}(f)}{\D_0(f)}}}^{\frac{\abs{\gamma}}{q}}\D_\gamma(f)$$
which completes the proof.
\end{proof}
\begin{cor}\label{cor:dissipation_interpolation_exp_solution}
Under the same conditions of Lemma \ref{lem:dissipation_interpolation_exp} we have that if $f(t,v)$ is a non-negative solution to \eqref{eq:BE} such that
$$\mathbf{\Gamma}_{a,q}^{\ast}=\sup_{t\geq 0}\mathbf{\Gamma}_{a,q}(f(t))<\infty $$
then
\begin{equation}\label{eq:dissipation_interpolation_exp_solution}
\D_\gamma(f(t)) \geq \frac{a^{\frac{\abs{\gamma}}{q}}\lambda_0 H(f(t)|\M)}{2\pa{\log\pa{\dfrac{2\mathbf{\Gamma}_{a,q}^{\ast}}{\lambda_0H(f(t)|\M)}}}^{\frac{\abs{\gamma}}{q}}}
\end{equation}
\end{cor}
\begin{proof}
{This follows immediately from Theorem \ref{thm:entropy_method_hard_potentials} and Lemma \ref{lem:dissipation_interpolation_exp}}
\end{proof}
In order to be able to conclude the desired rate of convergence to equilibrium we will need to {connect} $H_p(f(t)$ and $\mathbf{\Gamma}_{a,b}(f(t))$. To do so we notice the following:
\begin{lem}\label{lem:upper_bound_on_D_exp}
Let $a>0$, $p>1$ and $1<q\leq 2$ (with the additional assumption that, $a < 1/4$ whenever $q=2$). Then, for any non-negative function $f(v)$ we have that
{\begin{equation}\label{eq:upper_bound_on_D_exp}
\begin{gathered}
\mathbf{\Gamma}_{a,q}(f) \leq C_{a,q,p,d}\int_{\R^d}\exp\pa{2^{q}a\abs{v}^q}f(v)^p \d v\\
 -\norm{b}_\infty \int_{f\pa{v^\prime}\leq 1} \exp\pa{2^{q-1}a \abs{v}^q}f(v) \exp\pa{2^{q-1}a\abs{v_\ast}^q}\M(v_\ast)\log\pa{f\pa{v^\prime}}\d v\d \vb \d\sigma\end{gathered} 
\end{equation}}
for a uniform constant {$C_{a,q,p,d} >0$} that depends only on $a,q,p$ and $d$.
\end{lem}
\begin{proof}
We start by noticing that {since}
${\abs{v-v_\ast}^q} \leq {2^{q-1}\pa{\abs{v}^q+\abs{v_\ast}^q}}$
{and} $\abs{v-v_\ast}=\abs{v^\prime-v_\ast^\prime}$ {we have that}, for any $v,\vb,\sigma \in \R^{d}\times \R^{d}\times \sph$:
$$\exp\pa{a\abs{v-v_\ast}^q} \leq \min \pa{\exp\pa{2^{q-1}a\pa{\abs{v}^q+\abs{v_\ast}^q}},\exp\pa{2^{q-1}a\pa{\abs{v^\prime}^q+\abs{v^\prime_\ast}^q}}}.$$
Next, since
{\begin{equation*}
\mathbf{\Gamma}_{a,q}(f) \leq \norm{b}_{\infty} \int \exp\pa{a\abs{v-v_\ast}^q}f(v)\M(\vb) \pa{\log \pa{\frac{f(v)}{\M(v)}}-\log \pa{\frac{f(v')}{\M(v')}}}\d v\d \vb \d\sigma
\end{equation*}}
we see that
\begin{multline*}
\mathbf{\Gamma}_{a,q}(f) \leq \norm{b}_\infty \abs{\sph}C_p\int_{f(v)\geq 1} \exp\pa{2^{q-1}a \abs{v}^q}f(v)^p\exp\pa{2^{q-1}a\abs{v_\ast}^q}\M(v_\ast)\d v\d \vb \\
+\norm{b}_\infty \abs{\sph}\int_{\R^{2d}}f(v) 
\pa{\frac{d \log\pa{2\pi}}{2}+\frac{\abs{v}^2}{2}} \exp\pa{2^{q-1}a\pa{\abs{v_\ast}^q+\abs{v}^q}}\M(v_\ast)\d v\d \vb \\
-\norm{b}_\infty \int_{f\pa{v^\prime}\leq 1}f(v)\exp\pa{2^{q-1}a\pa{\abs{v_\ast}^q+\abs{v}^q}}\M(v_\ast)\log\pa{f\pa{v^\prime}}\d v\d \vb \d\sigma\end{multline*}
where we discarded the term involving $\log \M(v')$ which is nonpositive.
Under the additional requirement that {$2^{q-1}a<\frac{1}{2}$} if $q=2$ we see that we can find a constant $C_{a,q,p,d} >0$ that depends only on $a,q,p$ and $d$ such that
\begin{multline*}
\mathbf{\Gamma}_{a,q}(f) \leq C_{a,q,p,d}\int_{f(v)\geq 1} \exp\pa{2^{q}a\abs{v}^q}\pa{f(v)^p+f(v)}\d v\\
-\norm{b}_\infty \int_{f\pa{v^\prime}\leq 1} \exp\pa{2^{q-1}a \abs{v}^q}f(v) \exp\pa{2^{q-1}a\abs{v_\ast}^q}\M(v_\ast)\log\pa{f\pa{v^\prime}}\d v\d\vb\d\sigma\\\leq 2C_{a,q,p,d}\int_{\R^d}\exp\pa{2^{q}a\abs{v}^q}f(v)^p \d v\\
-\norm{b}_\infty \int_{f\pa{v^\prime}\leq 1} \exp\pa{2^{q-1}a \abs{v}^q}f(v) \exp\pa{2^{q-1}a\abs{v_\ast}^q}\M(v_\ast)\log\pa{f\pa{v^\prime}}\d v\d\vb\d\sigma
\end{multline*}
which {concludes} the proof.
\end{proof}
\begin{cor}\label{cor:upper_bound_on_D_exp_with_lower_bound_on_f}
Let $p>1$ and $0<a< \min\pa{\frac{1}{8},\frac{p-1}{8p}}$. Then for any non-negative $f$ such that
$$f(v) \geq A \exp\pa{-B\abs{v}^2},$$
for some $A,B>0$ we have that
\begin{equation}\label{eq:upper_bound_on_D_exp_lower_bound_on_f}
\begin{gathered}
\mathbf{\Gamma}_{a,2}(f) \leq C_{a,p,d}\pa{H_p(f)+\pa{\abs{\log A}+2B}H_p(f)^{\frac{1}{p}}}.
\end{gathered} 
\end{equation}
for a uniform constant $C_{a,p,d}$ that depends only on $a,p,\norm{b}_\infty$ and $d$.
\end{cor}
\begin{proof}
As $a<\frac{p-1}{8}$ we find that 
$$\int_{\R^d}\exp\pa{4a\abs{v}^2}f(v)^p\d v \leq \pa{2\pi}^{\frac{d}{2}}H_p(f).$$
Next, due to the lower bound on $f$ we find that
$$-\log f\pa{v^\prime}  \leq -\log A +B\abs{v^\prime}^2 \leq \abs{\log A}+B\pa{\abs{v}^2+\abs{v_\ast}^2}$$
and as such
\begin{multline*}
-\int_{f\pa{v^\prime}\leq 1} \exp\pa{2a \abs{v}^2}f(v) \exp\pa{2a\abs{v_\ast}^2}\M(v_\ast)\log f\pa{v^\prime}\d v\d\vb \d\sigma\\
\leq \pa{\abs{\log A}+2B}C_{a,d}\int_{R^d\times \R^d}\exp\pa{4a \abs{v}^2}f(v) \exp\pa{4a\abs{v_\ast}^2}\M(v_\ast)\d v\d\vb \\
=\pa{\abs{\log A}+2B}C_{a,d} \int_{\R^d}\exp\pa{4a\abs{v}^2}\M(v)^{\frac{p-1}{p}}\M(v)^{\frac{1-p}{p}}f(v)\d v\\
\leq \pa{\abs{\log A}+2B}C_{a,d} \pa{\int_{\R^d}\exp\pa{\frac{4ap}{p-1}\abs{v}^2}\M(v)\d v}^{\frac{p-1}{p}}H_p(f)^{\frac{1}{p}}.\end{multline*}
The result follows {from Lemma \ref{lem:upper_bound_on_D_exp} since $a<\frac{p-1}{8p}$.}
\end{proof}
Lastly, {before proving Theorem \ref{thm:rate_of_convergence_exp}, we show the following simple lemma:}
\begin{lem}\label{lem:H_p_bounds_moments}
Let $f$ be a non-negative function and let $s\geq 0$. Then
\begin{equation}\label{eq:H_p_bounds_moments}
m_s(f) \leq m_{\frac{sp}{p-1}}\pa{\M}^{\frac{p-1}{p}}H_p(f)^{\frac{1}{p}}.
\end{equation}
\end{lem}
\begin{proof}
We have that
$$m_s(f)=\int_{\R^d}\abs{v}^s\M(v)^{\frac{p-1}{p}}\M(v)^{\frac{1-p}{p}}f(v)\d v \leq \pa{\int_{\R^d}\abs{v}^{\frac{sp}{p-1}}\M(v)\d v}^{\frac{p-1}{p}}H_p(f)^{\frac{1}{p}},$$
completing the proof.
\end{proof}
\begin{proof}[Proof of Theorem \ref{thm:rate_of_convergence_exp}]
Since $H_p(f_0)<\infty$ we know, due to Corollary \ref{cor:H_p_decreases} that 
$$H_p(f(t))\leq H_p(f_0) <\infty.$$
This implies, by Lemma \ref{lem:H_p_bounds_moments} that $f(t,v)$ has bounded moments of any order. Using this together with Theorem \ref{thm:instantaneous_creation_of_maxwellian_lower_bound} we conclude that for any $t_0>0$ we can find appropriate constants such that 
$$f(t,v) \geq A_1 \exp\pa{-B_1 \abs{v}^2}.$$
This, together with Corollary \ref{cor:dissipation_interpolation_exp_solution} and \ref{cor:upper_bound_on_D_exp_with_lower_bound_on_f} with the choice of $a=\frac{1-p}{16p}$ shows inequality \eqref{eq:almost_cerc_exp}. As $\D_\gamma(f(t)) = -\frac{\d}{\d t}H\pa{f(t)|\M}$ the aforementioend inequality implies the desired convergence for $t\geq t_0$. \\
 {We are only left to show the correct rate of decay for $t<t_0$. Since all the moments exist, we can use Theorem \ref{thm:rate_of_convergence} (since $H_p(f)$ controls $\norm{f}_{p}$) to find that for $t\leq t_0$.
$$H\pa{f(t)|\M} \leq C_3\pa{1+t}^{-1} \leq C_3 \pa{\sup_{t\leq t_0}\frac{\exp\pa{\lambda t^{\frac{1}{1+\frac{\abs{\gamma}}{2}}}}}{1+t}}\exp\pa{-\lambda t^{\frac{1}{1+\frac{\abs{\gamma}}{2}}}}$$
form some constant $C_3$. This, together with our rate of decay for $t>t_0$, concludes the proof.}
\end{proof}

\section{About the non cut-off case}
\label{sec:nonc}

In this final section we aim to discuss a few preliminary results for
the linear Boltzmann equation with soft potential \emph{and without
  angular cut off assumption}. More precisely, we will assume that
there exist two positive constants $c_{1} \geq c_{0} >0$ such that
\begin{equation}
  \label{eq:bnoncut}
  c_{0}|\theta|^{-(d-1)-\nu} \leq b(\cos\theta)
  \leq c_{1}|\theta|^{-(d-1)-\nu}, \qquad \nu \in (0,2).
\end{equation}
In this case, it is simple to check that
$$\int_{\sph}b(\cos\theta)\d\sigma=\infty.$$
The divergence of the above integral means that we are not able to
split our linear operator into a gain and loss parts.

However, the study of the non-linear Boltzmann equation for soft
potentials without cut-off \cite{gress}, and in particular the
spectral analysis its linearised version (see for instance
\cite{strain}), suggests that the long-time behaviour of the linear
Boltzmann equation should, for some range of the parameters
$\gamma, \nu$, be similar to the one of the Boltzmann equation for
\emph{hard potentials}. In particular, we will show in the next
subsection the existence of a spectral gap as soon as $\gamma+\nu>0$.

\subsection{Existence of a spectral gap}
We still assume here that $b(\cdot)$ satisfies \eqref{eq:bnoncut} and we denote by
$\dD(f)$ the Dirichlet form associated to {the linear Boltzmann operator,} $\bm{L}_{B}$:
$$\dD(f)=\frac{1}{2}\int_{\R^{d}\times\R^{d} \times \sph}b(\cos\theta)|v-\vb|^{\gamma}\M(v)\M(\vb)\pa{h(v)-h(v')}^{2}\d v\d\vb \d\sigma$$
where $h=\frac{f}{\M}$. 

{An adaptation of the approach appearing in \cite{strain} yields the following:} 
\begin{propo}
  \label{propo:spectral}
  For any $\varepsilon > 0$ there is an explicit constant
  $C=C(B,\varepsilon) > 0$ such that
  {$$\dD(f) \geq C\,\left\|f-\varrho_{f}\M\right\|_{L^{2}(\langle v
      \rangle^{\gamma +\nu-\varepsilon}\M^{-1})}^{2}.$$}
  In particular, if $\gamma + \nu > 0$ there is $\lambda > 0$ so that
  \begin{equation}
    \label{eq:spectralgap}
    \dD(f) \geq \lambda\|f-\varrho_{f}\M\|^{2}_{L^{2}(\M^{-1})},
  \end{equation}
  i.e. $\bm{L}_{B}$ admits a spectral gap of size $\lambda$ in the
  space $L^2(\M^{-1})$
\end{propo}

\begin{proof}
  As was mentioned earlier, this result is a direct adaptation of
  \cite[Proposition 3.1]{strain}. We sketch the proof here for
  completion.
  
  Using the fact that $\dD(f-\varrho_{f}\M)=\dD(f)$, one can assume
  without loss of generality that $\varrho_{f}=0.$ {Since}, from
  \eqref{eq:bnoncut}, there exists $c_{0}> 0$ such that
  $$b(\cos\theta) \geq c_{0}(\sin \theta/2)^{-(d-1)-\nu}$$
  {it suffices to prove} the result for
  \begin{equation}\label{eq:B}
    B(v-\vb,\sigma)=|v-\vb|^{\gamma}\,(\sin \theta/2)^{-(d-1)-\nu}.
  \end{equation}
  {For a given $v,\vb \in \R^{2d}$, and for $0 < \beta< d-1+\nu$ to be chosen later, we define the set}
  $$\mathbf{C}_{\beta}=\mathbf{C}_{\beta}(v,\vb)=\left\{\sigma \in \sph\,;\;(\sin \theta/2)^{-(d-1)-\nu}
    \geq |v-\vb|^{\beta}\right\}.$$
  {Since the set $\mathbf{C}_{\beta}$ is invariant under the transformation $\sigma \to -\sigma$ and $(v,\vb) \to (v',\vb')$}
{\begin{equation*}\begin{gathered}
\dD(f) \geq \frac{1}{2}\int_{\R^{d}\times\R^{d}}\d v\d\vb \int_{\mathbf{C}_{\beta}(v,\vb)}b(\cos\theta)|v-\vb|^{\gamma}\M(v)\M(\vb)\pa{h(v)-h(v')}^{2}\d v\d\vb \d\sigma\\
\geq \frac{1}{2}\int_{\R^{d}\times \R^{d}\times \mathbf{C}_{\beta}}|v-\vb|^{\gamma+\beta}\M(v)\M(\vb)\pa{h(v)-h(v')}^{2}\d v\d\vb \d\sigma\\
=\int_{\R^{2d}\times \mathbf{C}_{\beta}}|v-\vb|^{\gamma+\beta}\M(v)\M(\vb)h^{2}(v)\d v\d\vb \d\sigma\\
 -\int_{\R^{2d}\times \mathbf{C}_{\beta}} |v-\vb|^{\gamma+\beta}\,\M(v)\M(\vb)h(v)h(v')\d v\d\vb \d\sigma=\mathcal{D}_{1}-\mathcal{D}_{2}.
 \end{gathered}
 \end{equation*}}
Now,
\begin{equation*}\begin{split}
\mathcal{D}_{1}&=\int_{\R^{2d}\times \mathbf{C}_{\beta}}|v-\vb|^{\gamma+\beta}\M(v)\M(\vb)h^{2}(v)\d v\d\vb \d\sigma\\
&=\int_{\R^{d}}f^{2}(v)\M^{-1}(v)\d v \int_{\R^{d}}\M(\vb)|v-\vb|^{\gamma+\beta}\d\vb \int_{\mathbf{C}_{\beta}}\d\sigma\end{split}\end{equation*}
{As (see \cite{strain}) there is some universal constant $c=c_{d} >0$ such that
$$\int_{\mathbf{C}_{\beta}}\d\sigma \geq c\,|v-\vb|^{-\frac{\beta(d-1)}{\nu+d-1}}$$
for any given $v,\vb \in \R^{d}$, we find that}
$$\mathcal{D}_{1} \geq c\,\int_{\R^{2d}}f^{2}(v)\d v \int_{\R^{d}}\M(\vb)\,|v-\vb|^{\gamma+\frac{\beta\nu}{\nu+d-1}}\d\vb$$
which, according to {Lemma \ref{lem:loss_operator}} yields the existence of some explicit constant $C_{\nu,\beta} >0$ such that
$$\mathcal{D}_{1} \geq C_{\nu,\beta} \int_{\R^{d}}f^{2}(v)\langle v\rangle^{\gamma+\frac{\beta\nu}{\nu+d-1}}\d v.$$
{Next, we notice that
\begin{equation*}\begin{gathered}
|\mathcal{D}_{2}| \leq \int_{\R^{2d}\times \sph}|v-\vb|^{\gamma+\beta}\M(v)\M(\vb) h(v)h(v')\d v\d\vb\\
=\int_{\R^{2d}\times \sph}|v-\vb|^{\gamma+\beta}f(v)\M(\vb)\M^{-1}(v')f(v')\d v\d\vb\\
=\int_{\R^{d}}f(v)\M^{-1}(v)\,\bm{K}_{\gamma+\beta}f(v)\d v
\end{gathered}\end{equation*}
where $\bm{K}_{\gamma+\beta}$ is the gain operator of the linear Boltzmann operator associated to the (cut-off) kernel $B(v-\vb,\cos\theta)=|v-\vb|^{\gamma+\beta}$. Recalling that (see Lemma \ref{lem:little_k_repres})}
$$\bm{K}_{\gamma+\beta}f(v)=\int_{\R^{d}}k_{\gamma+\beta}(v,w)f(w)\d w.$$
{we find that}
$$|\mathcal{D}_{2}| \leq \int_{\R^{2d}}f(v)\M^{-1}(v)k_{\gamma+\beta}(v,w)f(w)\d v \d w$$
{Following \cite{strain} again, one can show that }
$$|\mathcal{D}_{2}| \leq C\int_{\R^{d}}f(v)^{2}\langle v \rangle^{\gamma+\beta-(d-1)} \M^{-1}(v)\d v$$
as soon as $d-1 < \beta < d-1+\nu$. {Since this condition on $\beta$ implies that
$$\gamma  < \gamma+\beta-(d-1) < \gamma  + \frac{\beta\nu}{\nu+d-1}$$
we find that for any $\delta >0$ there exists $C_\delta>0$ such that}
$$|\mathcal{D}_{2}| \leq C_{\delta}\int_{\R^{d}}f^{2}(v)\langle v \rangle^{\gamma}\,\M^{-1}(v)\d v +\delta\int_{\R^{d}}f^{2}(v)\langle v \rangle^{\gamma  + \frac{\beta\nu}{\nu+d-1}}\,\M^{-1}(v)\d v.$$
Therefore, {choosing} $\delta >0$ small enough, one gets
\begin{equation*}\dD(f) \geq \mathcal{D}_{1}-\mathcal{D}_{2} \geq C_{1}\int_{\R^{d}}f^{2}(v)\langle v\rangle^{\gamma+\frac{\beta\nu}{\nu+d-1}}\M^{-1}(v)\d v
-C_{2}\int_{\R^{d}}f^{2}(v)\langle v \rangle^{\gamma}\,\M^{-1}(v)\d v\end{equation*}
for some $C_{1},C_{2} >0.$ Since {in addition, one can show that (again, see \cite{strain})
$$\dD(f)  \geq C_{\gamma}\| f\|_{L^{2}(\langle v\rangle^{\gamma}\M^{-1})}^{2},$$ 
we conclude that
$$\dD(f) \geq C_{3}\int_{\R^{d}}f^{2}(v)\langle v\rangle^{\gamma+\frac{\beta\nu}{\nu+d-1}}\M^{-1}(v)\d v
$$}
for some explicit constant $C_{3}$ depending on $\gamma, \beta, \nu$. {At this point we will choose  
$$\beta = \pa{\nu+d-1}\pa{1-\frac{\epsilon}{\nu}}$$
for $\epsilon>0$ small enough, and conclude the desired result}
\end{proof}

 \appendix
 
\section{Basic properties of the linear Boltzmann equation}\label{app:BE}

We collect here some of the technical properties of the linear
Boltzmann operator used in the core of the text. We begin with the
proof of Lemma \ref{lem:lyapunov_functionals} given in the
Introduction:

\begin{proof}[Proof of Lemma \ref{lem:lyapunov_functionals}]
  The fact that $H_{\mathbf{\Phi}}(\cdot|\M)$ is a Lyapunov functional
  of \eqref{eq:BE} for \emph{any} convex function $\mathbf{\Phi}$ is a
  general property of stochastic semigroups.  A rigorous proof {can be
    found} in \cite{rudnicki}. We {will only provide} a formal proof
  of {this} property. Differentiating $H_{\mathbf{\Phi}}(f(t)|\M)$
  {under the flow of the equation} and denoting by $h=f/\M$, we find
  that
$$\frac{\d}{\d t}H_{\mathbf{\Phi}}(f(t)|\M)=\int_{\R^d}\partial_t f(t,v){\mathbf{\Phi}}'\pa{h(t,v)}\d v$$
where ${\mathbf{\Phi}}'$ denotes the derivative of $\mathbf{\Phi}$. Then
\begin{multline*}
\frac{\d}{\d t}H_{\mathbf{\Phi}}(f(t)|\M) 
=-\int_{\R^d\times\R^d\times\sph}B\pa{\abs{v-v_\ast},\sigma}\M(v)\M(v_\ast)\times \\
\times\pa{h(t,v)-h\pa{t,v^\prime}}
{\mathbf{\Phi}}'\pa{h(t,v)}\d v\d\vb\d\sigma\end{multline*}
where we have used the fact that $\M(v)\M(v_\ast)=\M\pa{v^\prime}\M\pa{v^\prime_\ast}$. Using the usual pre-post collision change of variables yields 
\begin{multline*}\frac{\d}{\d t}H_{\mathbf{\Phi}}(f(t)|\M)=-\frac{1}{2}\int_{\R^d\times\R^d\times\sph}B\pa{\abs{v-v_\ast},\sigma}\M(v)\M(v_\ast)\times\\
\times \pa{h(t,v)-h(t,v^\prime)}\pa{{\mathbf{\Phi}}'\pa{h(t,v)}-{\mathbf{\Phi}}'\pa{h\pa{t,v^\prime}}}\d v\d\vb\d\sigma.\end{multline*}
The latter is nonnegative due to the convexity of $\mathbf{\Phi}$.
\end{proof}

\subsection{Carleman's representation}\label{app:carleman}
We now recall the Carleman's representation (see \cite{carleman,Vil}) of the gain operator for general interactions which we used in Section \ref{sec:maxwellian_lower_bounds}:
\begin{lem}\label{lem:gain_operator}
For any $\alpha \in\R$ the gain operator $\bm{K}_{\alpha}=\Q_{\alpha}^{+}(\cdot,\M)$ can be written as
\begin{equation}\label{eq:gain_operator}
\bm{K}_\alpha f(v)=\int_{\R^d} k_\alpha(v,w)f(w)\d w,
\end{equation}
where, for any $v,w \in \R^{2d}$, 
\begin{equation}\label{eq:def_little_k}
k_\alpha(v,w)=2^{d-1} \abs{v-w}^{-1}
\int_{\pa{v-w}^{\perp}}\abs{z-(v-w)}^{\alpha-(d-2)}
b\pa{\frac{\abs{z}^2-\abs{v-w}^2}{\abs{z}^2+\abs{v-w}^2}}\M(z+v)\d\pi(z)
\end{equation}
{with $(v-w)^{\perp}$ denoting} the hyperplane orthogonal to $(v-w)$ and $\d\pi(z)$ is the Lebesgue measure on that hyperplane. Moreover,
$$k_{\alpha}(v,w)\M(w)=k_{\alpha}(w,v)\M(v) \qquad \forall v,w \in \R^{d}\times \R^{d}.$$
\end{lem}
\begin{proof}
We start by recalling Carleman representation (see \cite[Appendix C]{gpv} for the derivation of the present expression): for a given interaction kernel $B(v-v_{\ast},\sigma)$ and given measurable functions $f,g$
\begin{multline}\label{eq:Carleman}
\int_{\R^d\times \sph} B\pa{v-v_\ast,\sigma}f\pa{v_\ast^\prime}g\pa{v^\prime}\d\vb\d\sigma \\
=2^{d-1}\int_{\R^d}\frac{f(w)}{\abs{v-w}} \d w \int_{E_{v,w}}\frac{B\pa{2v-z-w,\frac{z-w}{\abs{z-w}}}g(z)}{\abs{2v-z-w}^{d-2}}\d\pi(z),
\end{multline}
where $E_{v,w}$ is the hyperplane that passes through $v$ and is perpendicular to $v-w$. Applying this to $B(v-v_{\star},\sigma)=|v-v_{\ast}|^{\alpha}\,b(\cos \theta)$, one notes that, due to symmetry (recall that $b$ is even), it holds
\begin{equation*}\begin{split}
\bm{K}_\alpha f(v) &= 2^{d-1}\int_{\R^d}\frac{f(w)}{\abs{v-w}}\d w\int_{E_{v,w}}\frac{B\pa{2v-z-w,-\frac{z-w}{\abs{z-w}}}\M(z)}{\abs{2v-z-w}^{d-2}}\d\pi(z) \\
&=2^{d-1}\int_{\R^d}\frac{f(w)}{\abs{v-w}}\d w\int_{\pa{v-w}^{\perp}}\frac{\abs{v-z-w}^\alpha b\pa{\frac{v-z-w}{\abs{v-z-w}}\cdot\frac{w-z-v}{\abs{w-z-v}}}\M(z+v)}{\abs{v-z-w}^{d-2}}\d\pi(z)\\
&=2^{d-1}\int_{\R^d}\frac{f(w)}{\abs{v-w}}\d w\int_{\pa{v-w}^{\perp}}\abs{v-z-w}^{\alpha-(d-2)}b\pa{\frac{\abs{z}^2-\abs{v-w}^2}{\abs{z}^2+\abs{v-w}^2}}\M(z+v)\d\pi(z),\end{split}\end{equation*}
where we used the fact {that $z\perp (v-w)$ in the one before last expression}.\\
This proves \eqref{eq:gain_operator} and \eqref{eq:def_little_k}. In addition, for any $z \in (v-w)^{\perp}$, we  have $\ap{z,v}=\ap{z,w}$, which implies that
$$\abs{z+v}^2+\abs{w}^2=\abs{z+w}^2+\abs{v}^2.$$
Thus, on $(v-w)^\perp$ we have that $\M(z+w)\M(v)=M(z+v)\M(w)$. This, together with \eqref{eq:def_little_k}, shows that
$k_\alpha(v,w)\M(w)=k_\alpha(w,v)\M(v).$
\end{proof}

\begin{rem}\label{rem:Dkvw}
{We would like to point out at this point} that the above representation of the gain part allows to to establish  an alternative form of the entropy production associated to a convex mapping $\mathbf{\Phi}\::\:\R^{+}\mapsto \R^{+}$. {Indeed},
for any $\alpha > -d$, let $\D^{\mathbf{\Phi}}_{\alpha}$ be the associated $\mathbf{\Phi}$-entropy production of $\bm{L}_{\alpha}$:
$$\D^{\mathbf{\Phi}}_{\alpha}(f)=-\int_{\R^{d}}\bm{L}_{\alpha}f(v)\mathbf{\Phi}'\pa{\frac{f(v)}{\M(v)}}\d v.$$
Then, one can prove easily that
$$\D^{\mathbf{\Phi}}_{\alpha}(f)=\frac{1}{2}\int_{\R^{d}\times\R^{d}}k_{\alpha}(v,w)\M(w)\pa{h(v)-h(w)}\pa{\mathbf{\Phi}'(h(v))-\mathbf{\Phi}'(h(w))}\d v\d w$$
where $h=f/\M$ and $\mathbf{\Phi}'$ denotes the derivative of $\mathbf{\Phi}.$

{As the above above expression is actually valid for any $\alpha \in \R$, one can use it to give an alternative proof to the interpolation inequality \eqref{eq:DgD0Dmu} by showing that for $\gamma\in (-d,0)$ and $\mu>0$,
\begin{equation*}\label{eq:interpolation_of_k}
k_0(v,w)\leq k_\gamma(v,w)^{\frac{\mu}{\mu-\gamma}}k_{\mu}\,(v,w)^{-\frac{\gamma}{\mu-\gamma}}.
\end{equation*}
holds for a.e. $v,w \in \R^{d}$.}
\end{rem} 

{With the representation of $k_{\gamma}$ at hand, we can now show Lemma \ref{lem:little_k_repres}}. The proof is a simple adaptation of a similar study in \cite{carleman,MMK}
\begin{proof}[Proof of Lemma \ref{lem:little_k_repres}]
We start by writing {$V=\frac{v+\omega}{2}$ as $V=V_0+V_\perp$, where $V_\perp$ is the projection of $V$ on $\pa{v-\omega}^{\perp}$} and $V_0$ is parallel to $v-w$. Then $v=V+\frac{v-w}{2}$
and for any $z\in \pa{v-w}^\perp$
\begin{equation*}\begin{split}
\abs{v+z}^2&= \abs{\pa{V_0+\frac{v-w}{2}}+\pa{V_\perp+z} }^2 = \abs{V_0+\frac{v-w}{2}}^2+\abs{V_\perp+z}^2\\
&=\abs{V_0}^2+V_0\cdot (v-w)+\frac{\abs{v-w}^2}{4}+\abs{V_\perp+z}^2.\end{split}\end{equation*}
As
$$\frac{\abs{v}^2-\abs{w}^2}{2}=V\cdot (v-w)=V_0\cdot (v-w)=\pm\abs{V_0}\abs{v-w}$$
we can conclude that
$$\abs{V_0}^2=\frac{\pa{\abs{v}^2-\abs{w}^2}^2}{4\abs{v-w}^2}.$$
Thus, 
\begin{equation*}\begin{split}\abs{v+z}^2&=\frac{\pa{\abs{v}^2-\abs{w}^2}^2}{4\abs{v-w}^2}+\frac{\abs{v}^2-\abs{w}^2}{2}+\frac{\abs{v-w}^2}{4}+\abs{V_\perp+z}^2\\
&=\frac{1}{4}\pa{\abs{v-w}+\frac{\abs{v}^2-\abs{w}^2}{\abs{v-w}}}^2+\abs{V_\perp+z}^2,\end{split}\end{equation*}
which completes the proof.
\end{proof}

\section{Slow convergence to equilibrium}\label{app:slow}

In this Appendix we show that the rate of convergence to equilibrium
in \eqref{eq:BE} is naturally prescribed by the tails of the initial
datum $f_{0}$. Our main result is a simple adaptation of the analogue
Theorem from \cite{Lu} for the non-linear Boltzmann equation:
\begin{thm}
  Let $f_{0} \in L^{1}(\R^{d})$ be a non-negative initial datum with
  unit mass and let $f(t,\cdot)$ denotes the {solution} to
  \eqref{eq:BE}. For any $k \geq 0$, there exist {explicit} constants
  $C_{1} > 0$ and $C_{2,k} > 0$ such that
  $$\|f(t)-\M\|_{L^{1}_{k}} \geq C_{1}\int_{|v| > t^{\frac{1}{|\gamma|}}} \langle v \rangle^{k}f_{0}(v)\d v -C_{2,k}\exp\pa{-\frac{t^{\frac{2}{|\gamma|}}}{4}} \qquad \forall t \geq 0.$$
\end{thm}

\begin{proof}
  Using Duhamel's formula, one has, for a given $t > 0$,
$$f(t,v)=\exp\pa{-\Sigma_{\gamma}(v)t}f_{0}(v) + \int_{0}^{t}\bm{K}_{\gamma}f(s,v)\exp\pa{-\Sigma_{\gamma}(v)(t-s)} \d s \qquad \text{ for a.e. } v \in \R^{d}.$$
In particular, since $f(t,\cdot)$ is nonnegative
$$f(t,v) \geq \exp\pa{-\Sigma_{\gamma}(v)t}f_{0}(v) \qquad \text{ for a.e. } v \in \R^{d}, \qquad t  > 0.$$ 
Using {the fact that} $\Sigma_{\gamma}(v) \leq c_{\gamma}(1+|v|)^{\gamma} \leq c_{\gamma}|v|^{\gamma}$ for any $v \in \R^{d}$, one gets
$$f(t,v) \geq \exp\pa{-c_{\gamma}|v|^{\gamma}t}f_{0}(v) \qquad \text{ for a.e. } v \in \R^{d}, \qquad t > 0$$
and, in particular, setting $\alpha=1/|\gamma|$, one sees that
$$f(t,v) \geq \exp\pa{-c_{\gamma}}f_{0}(v) \qquad \text{ for a.e. } |v| > t^{\alpha}.$$
Consequently,
\begin{equation*}\begin{gathered}
\|f(t)-\M\|_{L^{1}_{k}} \geq \int_{|v| > t^{\alpha}}|f(t,v)-\M(v)| \langle v\rangle^{k}\d v \\
\geq \int_{|v| > t^{\alpha}}\langle v \rangle^{k}f(t,v)\d v -\int_{|v|>t^{\alpha}} \langle v \rangle^{k}\M(v)\d v\end{gathered}
\end{equation*}
{$$ \geq \exp\pa{-c_{\gamma}}\,\int_{|v| > t^{\alpha}} \langle v\rangle^{k}f_{0}(v) \d v -\pa{2\pi}^{-\frac{d}{2}}\int_{|v|>t^{\alpha}}
\langle v\rangle^{k}\exp\pa{-\frac{|v|^{2}}{2}}\d v.$$}
Since 
$$\ds\int_{|v| > t^{\alpha}}\langle v\rangle^{k}\exp\pa{-\frac{|v|^{2}}{2}}d v \leq \exp\pa{-\frac{t^{2\alpha}}{4}}\int_{\R^{d}}\langle v\rangle^{k}\exp\pa{-\frac{|v|^{2}}{4}}d v$$
{the proof is complete.}
\end{proof}

\bibliographystyle{plain}

\end{document}